\documentclass[11pt]{amsart}
\usepackage{amssymb,latexsym,amsmath,amsfonts,amsthm}
\usepackage{amssymb, amscd}
\usepackage{mathrsfs}
\usepackage{eucal}

\usepackage{marginnote}

\usepackage{color}
\usepackage{enumerate}
\usepackage{soul}
\usepackage{xcolor}

\usepackage{mathtools}

\numberwithin{equation}{section}

\renewcommand{\pmb}{\boldsymbol}
\usepackage{graphicx}
\usepackage{amscd}

\usepackage{times}

\usepackage{fullpage}

\DeclareSymbolFont{SY}{U}{psy}{m}{n}
\DeclareMathSymbol{\emptyset}{\mathord}{SY}{'306}

\theoremstyle{plain}

\newtheorem{thm}{Theorem}[section]
\newtheorem{cor}[thm]{Corollary}
\newtheorem{lem}[thm]{Lemma}
\newtheorem{prop}[thm]{Proposition}

\theoremstyle{definition}
\newtheorem{defn}[thm]{Definition}

\newtheorem{rem}[thm]{Remark}

\newtheorem{qn}[thm]{Question}

\newcommand{\norm}[1]{\left\lVert#1\right\rVert}
\newcommand{\overbar}[1]{\mkern 1.5mu\overline{\mkern-1.5mu#1\mkern-1.5mu}\mkern 1.5mu}

\newcounter{defcounter}
\title[curvature inequalities]{Curvature inequalities and extremal operators}
\author{Gadadhar Misra}
\author{Md. Ramiz Reza}
\address{Gadadhar Misra: Department of Mathematics, Indian Institute of Science, Bangalore}
\email{gm@iisc.ac.in}

\address{Md. Ramiz Reza: Department of Mathematics, Indian Institute of Science, Bangalore}
\email{ramiz.md@gmail.com}

\begin{document}


\thanks{The work of  G. Misra was supported, in part,  the J C Bose National Fellowship and the UGC, SAP -- CAS. The work of Md. Ramiz Reza was supported, in part, by a Senior Research Fellowship of the CSIR. The results of this paper are from his PhD thesis submitted to the Indian Institute of Science.}

\subjclass[2010]{30C40, 47A13, 47A25} 
\keywords{Cowen-Douglas class, curvature inequality, extremal operators, S\"{z}ego kernel, holomorphic hermitian vector bundle}

\begin{abstract}
A curvature inequality is established for contractive commuting tuples of operators $\boldsymbol T$ in the Cowen-Douglas class $B_n(\Omega^*)$ of rank $n$ defined on some bounded domain $\Omega$ in $\mathbb C^m.$ 
Properties of the extremal operators, that is, the operators which achieve equality, are investigated. Specifically, a substantial part of a well known question due to R. G. Douglas involving these extremal operators, in the case of the unit disc, is answered. 

\end{abstract}
\maketitle
\section{Introduction}

For a fixed $n\in\mathbb N,$ and a bounded domain $\Omega\subseteq \mathbb C^m,$ 
the important class of operators $B_n(\Omega^*),$  $\Omega^*=\{\bar{z}: z\in \Omega\},$ defined below, was introduced in the papers \cite{CD} and \cite{Douglasopen} by Cowen and Douglas. An alternative approach to the study of this class of operators is presented in the paper  \cite{Curtosalinas} of Curto and Salinas. For $w=(w_1,w_2,\ldots, w_m)$ in $\Omega^*,$   let $\mathcal{D}_{\pmb{T}-w\pmb{I}}: \mathcal H \to \mathcal H \oplus \mathcal H \oplus \cdots \oplus \mathcal H$ be the operator: $\mathcal{D}_{\pmb{T}-w\pmb{I}}(h) = \oplus _{k=1}^{m} (T_k- w_k I)h,$ $h\in \mathcal H.$ 
\begin{defn}
A $m$-tuple of commuting bounded operators $\pmb{T} = (T_1,T_2,\ldots,T_m)$ on a complex separable Hilbert space $\mathcal H$ is said to be in $B_n(\Omega^*)$ if 
\begin{enumerate}
\item $\dim \big (\cap _{k=1}^{m} \text{ker}(T_k -w_kI) \big ) = n$ for each  $w \in \Omega^*;$ 
\item the operator $\mathcal{D}_{\pmb{T}-w\pmb{I}},$  $w \in \Omega^*,$  has closed range and 
\item $\bigvee_{w\in \Omega^*} \big(\cap_{k=1}^{m} \text{ker}(T_k -w_kI)\big) = \mathcal{H}$
\end{enumerate}
\end{defn}
For any commuting tuple of operators $\pmb{T}$ in $B_n(\Omega^*),$ the existence of a rank $n$ holomorphic Hermitian  vector bundle $E_{\pmb T}$ over $\Omega^*$ was established in \cite{Douglasopen}. Indeed, 
$$E_{\pmb T} := \big \{(w,v)\in \Omega^* \times \mathcal{H} : v\in\cap_{k=1}^{m} \ker(T_k -w_kI)\big \}, \pi (w,v)=w,$$
admits a local holomorphic cross-section. 
In the paper \cite{CD}, for $m=1,$ it is shown that two commuting $m$-tuple of operators $\pmb T$ and $\pmb S$ in $B_n(\Omega^*)$ are jointly unitarily equivalent if and only if  $E_{\pmb T}$  and $E_{\pmb S}$ are locally equivalent as  holomorphic Hermitian vector  bundles. This proof works for the case $m > 1$ as well. 

Suppose $\mathscr{K} = \mathscr{K}(E_{\pmb T} , D)$ is the curvature associated with canonical connection $D$ of the holomorphic Hermitian vector bundle $E_{\pmb T}.$ Then relative to any $C^{\infty}$ cross-section  $\sigma$ of   $E_{\pmb T},$ we have 
\begin{equation*}
\mathscr{K}(\sigma) = \sum\limits_{i,j=1}^{m}\mathcal{K}^{i,j} (\sigma)\; dz_i \wedge d\bar{z}_j,\;\;
\end{equation*} where each $\mathcal{K}^{i,j}$ is a $C^{\infty}$ cross-section of  $\text{Hom}(E_{\pmb T} , E_{\pmb T})$.
Let $\pmb{\gamma}(z) = (\gamma_1(z),\ldots,\gamma_n(z))$ be a local holomorphic frame of $E_{\pmb T}$ in a neighbourhood $\Omega^*_0\subset \Omega^*$ of some $w\in \Omega^*.$ The metric of the bundle $E_{\pmb T}$ at $z\in \Omega^*_0$ w.r.t the frame $\pmb{\gamma}$ has the matrix representation
\begin{align*}
h_{\pmb{\gamma}}(z) &= \big(\!\!
           \big( \langle \gamma_j(z), \gamma_i(z) \rangle \big)
          \!\!\big)_{i,j = 1}^{n}.
\end{align*}
We write $\partial_i= \frac{\partial}{\partial z_i}$ and $\bar{\partial}_i= \frac{\partial}{\partial \bar{z}_i}.$ The  co-efficients of the curvature $(1,1)$-form $\mathscr K$ w.r.t the frame $\pmb{\gamma},$  are explicitly determined by the formula: 
\begin{align*}
\mathcal{K}^{i,j}(\pmb{\gamma})(z) &= - \bar{\partial _j} \bigg((h_{\pmb{\gamma}}(z))^{-1}\big(\partial _i h_{\pmb{\gamma}}(z)\big)\bigg),\;z\in \Omega^*_0.
\end{align*}


For a bounded domain $\Omega$ in $ \mathbb C$ and for $T$ in $B_n(\Omega^*),$ recall that $N_w^{(k)}$ is the restriction of the operator $(T-wI)$ to the subspace $\ker(T-w I)^{k+1}.$  In general, even if $m=1,$ it is not possible to put  the operator $N_w^{(k)}$ into any reasonable canonical form, see \cite[sec. 2.19]{CD}. Here we show how to do this for any $m \in \mathbb N,$ assuming that $k=1.$ The canonical form of the operator $N_w^{(1)},$ we find here, is a crucial ingredient in obtaining the curvature inequality for a commuting tuple of operator $\pmb T$ in $B_n(\Omega^*),$ which admit $\overbar{\Omega}^*,$ the closure of $\Omega^*, $ as a spectral set.


A commuting $m$-tuple of operator $\pmb T$ in $B_n(\Omega^*),$ may be realized as the $m$-tuple  $\pmb M^*= (M_{z_1}^*,\ldots,M_{z_m}^*),$ the adjoint of the multiplication by the $m$ coordinate functions on some Hilbert space of holomorphic functions defined on $\Omega$ possessing a reproducing kernel $K$ (cf. \cite{CD, Curtosalinas}). The real analytic function $K(z,z)$ then serves as a Hermitian metric for the vector bundle $E_T$ w.r.t. the holomorphic frame $\gamma_i(\bar{z}):=K(\cdot ,z)e_i,$ $i=1, \ldots, n,$ $\bar{z}$ in some open  subset $\Omega_0^*$ of $\Omega^*.$ Here the vectors $e_i,$ 
 $i=1, \ldots, n,$ are the standard unit vectors of $\mathbb C^n.$ 
For a point $z\in \Omega,$ let $\mathscr{K}_{T}(\bar{z})$ be the curvature of the vector bundle $E_T.$ 
It is  easy to compute the co-efficients of the curvature $\mathscr{K}_{T}(\bar{z})$ explicitly using the metric $K(z,z)$ for $m=1, n=1,$ namely, 
 $$ \mathcal{K}^{i,j}_T(\bar{z}) = -\frac{\partial ^2}{\partial w_i \partial \bar{w}_j} \mbox{log} K(w,w)|_{w=z} =    -\frac{\norm{K_{z}}^2 \langle \bar{\partial}_i K_z , \bar{\partial_j}K_{z} \rangle - {\langle K_{z}, \bar{\partial_i}K_{z} \rangle} {\langle \bar{\partial_j}K_{z}, K_{z}\rangle}}{(K({z},{z}))^2},\;\;z \in \Omega. $$
Set $\mathcal K = \big ( \!\!\big ( \mathcal K^{i,j}\big )\!\!\big )$.  

First, consider the case of $m=1.$ Assume that $\overbar{\Omega}^{*}$ is 
a spectral set for an operator $T$ in $B_1(\Omega^*),$ $\Omega \subset \mathbb C.$ 
Thus for any rational function $r$ with poles off $\overbar{\Omega}^*,$ we have 
$\|r(T)\| \leq \|r\|_{\Omega^*, \infty}.$ For such operators $T,$ the curvature inequality 
\begin{align*}
\mathcal{K}_{T} (\bar{w}) &\leq -4\pi ^2(S_{\Omega^*}(\bar{w},\bar{w}))^2,\;\;\;\bar{w}\in\Omega^*,
\end{align*}
where $S_{\Omega^*}$ is the S\"{z}ego kernel of the domain $\Omega^*,$ was established in \cite{GMCI}.
Equivalently, since $S_{\Omega}(z,w) = S_{\Omega^*}(\bar{w},\bar{z}),\,\;z,w\in \Omega,$ the curvature inequality takes the form 
\begin{align}\label{eq: alternative CI}
\frac{\partial ^2}{\partial w \partial \bar{w}}\mbox{log} K_{T}(w,w)\geq 4\pi ^2 (S_{\Omega}(w,w))^2, \;\;\;w \in \Omega.
\end{align}

Let us say that a commuting tuple of operators $\boldsymbol{T}$ in $B_n(\Omega ^*),$ $\Omega \subset \mathbb C^m,$ is \emph{contractive} if  $\overbar{\Omega}^*$ is a spectral set for $\boldsymbol{T},$ that is, $\|f(\boldsymbol T)\| \leq \|f\|_{\Omega^*, \infty}$ for all functions holomorphic in some neighborhood of $\overbar{\Omega}^*$.   

In this paper, see Theorem \ref{Generalized CI}, we generalize the curvature inequality \eqref{eq: alternative CI} for a contractive tuple of operators $\boldsymbol T$ in $B_n(\Omega ^*),$ which include the earlier inequalities from  \cite{Misra90bdmodule,Misra93Pati}.   

Let $U_+$ be the forward unilateral shift operator on $\ell^2(\mathbb N).$ The adjoint  $U_+^*$ is the backward shift operator and is in $B_1(\mathbb D).$ The shift $U_+$ is unitarily equivalent to  the multiplication operator $M$ on the Hardy space $(H^2(\mathbb{D}), ds).$ The reproducing kernel of the Hardy space is the  S\"{z}ego kernel $S_{\mathbb{D}}(z,a)$ of the unit disc $\mathbb D.$ It is given by the formula  $S_{\mathbb{D}}(z,a) = \frac{1}{2\pi(1-z\bar{a})}, \,z,a \in \mathbb{D}.$ A straightforward computation gives an explicit formula for the curvature 
$\mathcal{K}_{U_+^*} (w):$ 
\begin{align*}
\mathcal{K}_{U_+^*} (w)=- \frac{\partial ^2}{\partial w \partial \bar{w}}\mbox{log} S_{\mathbb{D} }(w,w) &= -4 \pi ^2 (S_{\mathbb{D}}(w,w))^2,\;\;\;w\in \mathbb D.
\end{align*}
Since the closed unit disc is a spectral set for any contraction $T$ (by von Neumann inequality), it follows, from equation $\eqref{eq: alternative CI},$ that the curvature of the operator $U_+^*$ dominates the curvature of every other contraction $T$ in $B_1(\mathbb D),$  
\begin{align*}
\mathcal{K}_T(w) \leq \mathcal{K}_{U_+^*}(w)= - (1-|w|^2)^{-2},\;\;w\in \mathbb{D}.
\end{align*}
Thus the operator $U_+^*$ is the \emph{extremal operator} in the class of contractions in $B_1(\mathbb D).$ 
The extremal property of the operator $U_+^*$ prompts the following question due to R. G. Douglas.  

\begin{qn}[R. G. Douglas]\label{Douglas question}
For a contraction $T$ in $B_1(\mathbb D),$ if $\mathcal K_T(w_0) =- (1-|w_0|^2)^{-2}$ for some fixed $w_0$ in $\mathbb D,$ then does it follow that $T$ must be unitarily equivalent to the operator $U_+^*$?
\end{qn}
It is known that the answer is negative, in general,  however it has an affirmative answer if, for instance, $T$ is a homogeneous contraction in $B_1(\mathbb D),$ see \cite{GMback}. From the simple observation that $\mathcal{K}_T(\bar{\zeta})  = - (1-|\zeta|^2)^{-2}$ for some $\zeta\in \mathbb D$ if and only if the two vectors $\tilde{K}_{\zeta}$ and $\bar{\partial}\tilde{K}_{\zeta}$ are linearly dependent, where $\tilde{K}_w(z) = (1-z\bar{w}) K_w(z),$  it follows that the question of Douglas has an affirmative answer in the class of contractive, co-hyponormal backward weighted shifts. Furthermore, we obtain an affirmative answer if $T^*$ is assumed to be a $2$ hyper-contraction and that $1$ is a cyclic vector for ${\phi(T)}^*,$ where $\phi$ is the bi-holomorphic automorphism of $\mathbb D$ mapping $\zeta$ to $0.$  This is Theorem \ref{uniqueness in unit disc} of this paper. 

If the domain $\Omega$ is not simply connected, it is not known if there exists an extremal operator $T$ in $B_1(\Omega^*)$, that is, if  
$$\frac{\partial ^2}{\partial w \partial \bar{w}}\mbox{log} K_{T}(w,w)= 4\pi ^2 (S_{\Omega}(w,w))^2,\,\,w\in\Omega,$$ for some $T$ in  $B_1(\Omega^*)$. Indeed, from a result of Suita (cf. \cite{SUITA}), it follows that the adjoint of the multiplication operator on the Hardy space $H^2(\Omega)$ is not extremal. It was shown in \cite{GMCI} that for any fixed but arbitrary $w_0\in \Omega,$ there exists an operator $T$ in $B_1(\Omega^*)$  for which equality is achieved, at  $w=w_0,$ in the inequality \eqref{eq: alternative CI}. The question of the uniqueness of such an operator, which is the Douglas question, was partially answered recently by the second named author in \cite{MRuniqueness}. The precise result is that these ``point-wise'' extremal operators are determined uniquely within the class of the adjoint of the bundle shifts introduced in \cite{ABsubnormal}. It was also shown in the same paper that each of these bundle shifts can be realized as  a multiplication operator on a Hilbert space of weighted Hardy space and conversely. Generalizing these results, in this paper, we prove that  the local extremal operators are uniquely determined in a much larger class of operators, namely, the ones that includes all the weighted Bergman spaces along with the weighted Hardy spaces defined on $\Omega.$
This is Theorem \ref{uniqueness in finitely connected}. The authors have obtained some preliminary results in the multi-variate case which are not included here.


\section{Local operators and Generalized Curvature Inequality }

Let $\Omega$ be a bounded domain in $\mathbb{C}^m$ and $\pmb{T} = (T_1,T_2,\ldots,T_m)$ be a commuting $m$-tuple of  bounded operators on some separable complex Hilbert space $\mathcal{H}.$ Assume that the tuple of operator $\pmb T$ is in $B_n(\Omega^*).$ For an arbitrary but fixed point $w\in \Omega^*,$ let 
\begin{align}\label{eq:1.1a}
\mathcal{M}_w &=\bigcap \limits _{i,j=1}^{m}\;\;\ker(T_i-w_i)(T_j-w_j).
\end{align}
Clearly, the joint kernel $\cap_{i=1}^m \ker(T_i-w_i)$ is a subspace of $\mathcal{M}_w.$  Fix a holomorphic frame $\pmb{\gamma},$ defined on some neighborhood of $w,$ say $\Omega_0^* \subseteq \Omega^*,$  of the vector bundle $E_{\pmb T}.$ Thus $\pmb{\gamma}(z) =(\gamma_1(z), \ldots, \gamma_n(z)),$ for $z$ in $\Omega_0^*,$  for some choice $\gamma_i(z),$  $i = 1,2,\ldots,n,$ of joint eigenvectors, that is, $(T_j-z_j)\gamma_i(z) =0,$ $j = 1,2,\ldots,m.$ It follows that 
\begin{equation}\label{eq:action of N_j} 
 (T_j-w_j)(\partial_k\gamma_i(w)) = \gamma_i(w)\delta_{j,k},\;i = 1,2,\ldots,n, \;\;\mbox{and}\;\;j,k = 1,\ldots, m.
\end{equation}
The eigenvectors $\pmb{\gamma}(w)$ together with their derivatives, that is $(\partial_1 \pmb{\gamma}(w),\ldots,\partial_m \pmb{\gamma}(w))$ is a basis for the subspace $\mathcal M_w.$  

The metric of the bundle $E_{\pmb T}$ at $z\in \Omega^*_0$ w.r.t the frame $\pmb{\gamma}$ has the matrix representation
\begin{align*}
h_{\pmb{\gamma}}(z) &= \big(\!\!
           \big( \langle \gamma_j(z), \gamma_i(z) \rangle \big)
          \!\!\big)_{i,j = 1}^{n}.
\end{align*}
Clearly, $\pmb{\tilde{\gamma}}(z) = (\gamma_1(z), \ldots ,\gamma_n(z))h_{\pmb{\gamma}}(w)^{-1/2}$ is also a holomorphic frame for $E_{\pmb T}$ with the additional 
property that $\pmb{\tilde{\gamma}}$ is orthonormal at $w,$  that is, $h_{\pmb{\tilde{\gamma}}}(w)= I_n.$
We therefore assume, without loss of generality, that $h_{\pmb{\gamma}}(w)=I_n.$

In what follows, we always assume that we have made a fixed but arbitrary choice of a local holomorphic frame $\pmb{\gamma}(z) = (\gamma_1(z),\ldots,\gamma_n(z))$ defined on a small neighborhood of $w,$ say $\Omega_0^* \subseteq \Omega,$ 
such that $h_{\pmb{\gamma}}(w)=I_n.$

Recall that the local operator  $\pmb N_w = (N_1(w),\ldots,N_m(w))$ is the commuting $m$-tuple of nilpotent operators on the subspace $\mathcal M_w$ defined by $N_i(w) = (T_i-w_i)\mid _{\mathcal{M}_w}$. As a first step in relating the operator $\pmb T$ to the vector bundle $E_{\pmb T},$ pick a holomorphic frame $\pmb {\gamma},$ satisfying $h_{\pmb{\gamma}}(w)=I_n,$ for the holomorphic Hermitian vector bundle $E_{\pmb T}$ which also serves as a basis for the joint kernel of $\pmb T.$ We extend this basis to a basis of $\mathcal M_w.$  In the following proposition, we determine a natural orthonormal basis in $\mathcal M_w$ such that the curvature of the vector bundle $E_{\pmb T}$ appears in the matrix representation (obtained with respect to this orthonormal basis) of $\pmb N_w.$


\begin{prop} \label{prop1.1}
There exists an orthonormal basis in the subspace $\mathcal M_w$ such that the matrix representation of $N_l(w)$ with respect to this basis is of the form 
\begin{align*}
N_l(w) = \begin{pmatrix}
0_{n\times n} &  \pmb t_l(w)\\
0_{mn \times n} & 0_{mn \times mn}
\end{pmatrix},\,\,\begin{psmallmatrix}
\mathbf t_1(w)\\
\vdots\\
\mathbf t_m(w)
\end{psmallmatrix}
\begin{pmatrix} \overline{\mathbf t_1(w)}^t, \ldots, \overline{\mathbf t_m(w)}^t\end{pmatrix} =
 \pmb t(w) \overline{\pmb t(w)}^{\rm tr}= 
 -(\mathcal{K}_{\pmb{\gamma}}(w))^{-1},
\end{align*}
where ${\pmb{\gamma}}$ is a frame of $E_{\pmb T}$ defined in a neighborhood of $w$ which is orthonormal at the point $w$ and $\mathcal K_{{\pmb{\gamma}}}(z)= \big(\!\!\big (\mathcal{K}_{i,j}({\pmb{\gamma}})(z)\big )\!\!\big).$ 
\end{prop}
\begin{proof} 
For any $k= (p-1) n + q,$ $1 \leq p \leq m+1,$ and $1 \leq q \leq n$, set $v_k := \partial_{p-1}(\gamma_{q}(w))$ and $\mathbf v_i:= \big ( v_{(i-1)n +1}, \ldots ,  v_{(i-1)n + n} \big ).$ Thus $\mathbf v_i$ is also $\partial_{i-1}\pmb{\gamma},$
where $\pmb{\gamma} = (\gamma_1, \ldots , \gamma_n).$
Hence the set of vectors $\{v_k,$ $1 \leq k \leq (m+1)n\}$ forms a basis of the subspace $\mathcal M_w.$  Let $P$ be an invertible matrix of size $(m+1)n\times (m+1)n$ and $$(\mathbf u_1, \ldots , \mathbf u_{m+1}) := (\mathbf v_1, \ldots , \mathbf v_{m+1} )\begin{pmatrix}
 P_{1,1} & P_{1,2} & \ldots & P_{1,m+1}\\
 P_{2,1} & P_{2,2} & \ldots & P_{2,m+1}\\
 \vdots & \vdots & \ddots & \vdots \\
 P_{m+1,1} & P_{m+1,2} & \ldots & P_{m+1,m+1}
 \end{pmatrix},$$ 
where each $P_{i,j}$ is a $n\times n$ matrix. 
Clearly, $(\mathbf u_1, \ldots ,\mathbf u_{m+1})$ is a basis, not necessarily orthonormal, in the subspace $\mathcal M_w.$ 
The set of vectors $\{\mathbf{u}=(\mathbf u_1, \ldots \mathbf u_{m+1})\}$ is an orthonormal basis in $\mathcal{M}_w$ if and only if $P\bar{P}^t = G^{-1},$ where $G$ is the $(m+1)n\times (m+1)n,$ grammian $\big(\!\!\big( \langle v_j,v_i\rangle\big)\!\!\big),$ that is, 
\begin{align*}
 G &=  \begin{pmatrix}
      h_{\pmb{\gamma}}(w) & \partial_1 h_{\pmb{\gamma}}(w) & \ldots & \partial_m h_{\pmb{\gamma}}(w)\\
      \bar{\partial}_1 h_{\pmb{\gamma}}(w) &\bar{\partial}_1\partial_1 h_{\pmb{\gamma}}(w) & \ldots & \bar{\partial}_1\partial_m h_{\pmb{\gamma}}(w) \\
      \vdots  & \vdots  & \ddots & \vdots  \\
      \bar{\partial}_m h_{\pmb{\gamma}}(w) & \bar{\partial}_m\partial_1 h_{\pmb{\gamma}}(w) & \ldots & \bar{\partial}_m\partial_m h_{\pmb{\gamma}}(w)
      \end{pmatrix}.
\end{align*}  
In particular, we choose and fix  $P$ to be the upper triangular matrix corresponding to the Gram-Schmidt orthogonalization process. 
Following equation $\eqref{eq:action of N_j},$ the matrix representation of $N_l(w)$ w.r.t. the basis $\mathbf{v}= (\mathbf v_1, \ldots , \mathbf v_{m+1})$ is $[N_l(w)]_{\mathbf{v}} = \big (\!\! \big ( N_l(w)_{i j} \big ) \!\!\big ),$ $l= 1,2,\ldots ,m,$ where 
\begin{align*}
N_l(w)_{i j} = \begin{cases}
0_{n \times n} & (i,j) \not = (1, l+1)\\
I_{n} & (i,j) = (1,l+1) 
\end{cases}, \,\, 1\leq i, j \leq m+1.
\end{align*}

%
Therefore  w.r.t the orthonormal basis $(\mathbf u_1,\ldots, \mathbf u_{m+1}),$ the matrix of $N_l$ is of the form
\begin{align}\label{matrix of NL}
\begin{bmatrix}
N_l(w)
\end{bmatrix} _{\mathbf{u}} 
&=\begin{pmatrix}
                 0_{n \times n} & {t}_l ^{1}(w) & \ldots & {t}_l ^{m}(w)\\
                 0_{n \times n} & 0_{n \times n} & \ldots &  0_{n \times n}\\
                 \vdots & \vdots  & \ddots & \vdots\\
                  0_{n \times n}  & 0_{n \times n} & \ldots  & 0_{n \times n}
                \end{pmatrix}
                = \begin{pmatrix}
                0_{n \times n} & \mathbf{t}_l(w) \\
                0_{mn \times n} & 0_{mn \times mn}
                  \end{pmatrix},
\end{align}
where each ${t}_l^{i}(w)$ is a square matrix of size $n,$ for $l,i = 1,2,\ldots ,m$ and $\pmb t_l(w)$ is a $n \times mn$ rectangular matrix. It is now evident that for $l,r = 1,2,\ldots,m$, we have
\begin{align*}
\begin{bmatrix}
N_l(w) N_r(w) ^*
\end{bmatrix} _{\mathbf{u}} &= Q \begin{bmatrix}
                    N_l(w)
                    \end{bmatrix}_{\mathbf{v}} G^{-1}\begin{bmatrix}
                    N_r(w)
                    \end{bmatrix}_{\mathbf{v}} \overbar{Q}^t, 
\end{align*}
where $Q=P^{-1}.$ To continue, we write the matrix $G^{-1}$ in the form of a block matrix:
\begin{align} \label{eq: gram}
G^{-1} &= \begin{pmatrix}
         *_{n \times n} & *_{n \times n} & *_{n \times n} & \ldots & *_{n\times n}\\
           *_{n \times n} & R_{1,1} & R_{1,2} & \ldots & R_{1,m}\\
           *_{n \times n} & R_{2,1} & R_{2,2} & \ldots & R_{2,m}\\
           \vdots & \vdots & \vdots & \ddots & \vdots\\
           *_{n \times n} & R_{m,1} & R_{m,2} & \ldots & R_{m,m}
          \end{pmatrix}
          =\left (  \begin{array}{l l}
            *_{n \times n} & *_{n \times mn}\\
            *_{mn \times n} & R
            \end{array} \right ),
          \end{align} 
\renewcommand\arraystretch{1.0}
where each $R_{i,j}$ is a $n \times n$ matrix. Then we have 
 \begin{align*}
 \begin{bmatrix}
N_l(w) N_r(w) ^*
\end{bmatrix} _{\mathbf{u}}                 
%
                   &= \begin{pmatrix}
                       Q_{1,1}R_{l,r}{\overbar{Q}^t_{1,1}} & 0_{n \times mn}\\
                       0_{mn \times n} & 0_{mn \times mn}
                      \end{pmatrix} .    
\end{align*}
Since $P$ is upper triangular with $P_{1,1}=I_n,$ we have $\mathbf u_1 = \mathbf v_1 P_{1,1}= \mathbf v_1,$  that is, 
\begin{eqnarray*}
(u_1,u_2,\ldots , u_n) &=& (v_1,v_2,\ldots, v_n).
\end{eqnarray*}

As $P_{1,1}= I_n,$ we have $Q_{1,1} = I_n.$ Hence w.r.t the orthonormal basis $(\mathbf u_1,\ldots,\mathbf u_{m+1})$ of the subspace $\mathcal{M}_w$, the linear transformation $N_l(w)N_r(w)^*$ has the matrix representation 
 \begin{align}
[N_l(w)N_r(w)^*]_{\mathbf{u}} 
               &= \begin{pmatrix}
                  R_{l,a} & 0_{n \times mn}\\
                  0_{mn \times n} & 0_{mn \times mn}
                  \end{pmatrix}. \label{eq:1j}
 \end{align}
 Let $\pmb t(w)$ be the $mn \times mn$ matrix given by 
\begin{align*}
\pmb t(w)=\begin{psmallmatrix}
\mathbf t_1(w)\\
\mathbf t_2(w)\\
\vdots\\
\mathbf t_m(w)
\end{psmallmatrix}.
\end{align*} 
Now combining equation \eqref{matrix of NL} and equation \eqref{eq:1j}, we then have 
\begin{align}\label{tw and its adjoint}
\pmb t(w) \overline{\pmb t(w)}^{\rm tr} = R. 
\end{align}
To complete the proof, we have to relate the block matrix $R$  to the curvature matrix $\mathcal{K}_{\pmb{\gamma}}(w)$ w.r.t the frame $\pmb{\gamma}$. Recalling $\eqref{eq: gram},$ we have that 
\begin{align*}
G^{-1} &= \begin{pmatrix}
            *_{n \times n} & *_{n \times mn}\\
            *_{mn \times n} & R
            \end{pmatrix}.
\end{align*}

The Grammian $G$ admits a natural decomposition as a  $2 \times 2$ block matrix, namely, 
\begin{eqnarray*}
 G &=&  \begin{pmatrix}
      h_{\pmb{\gamma}}(w) & \partial_1 h_{\pmb{\gamma}}(w) & \ldots & \partial_m h_{\pmb{\gamma}}(w)\\
      \bar{\partial}_1 h_{\pmb{\gamma}}(w) &\bar{\partial}_1\partial_1 h_{\pmb{\gamma}}(w) & \ldots & \bar{\partial}_1\partial_m h_{\pmb{\gamma}}(w) \\
      \vdots  & \vdots  & \ddots & \vdots  \\
      \bar{\partial}_m h_{\pmb{\gamma}}(w) & \bar{\partial}_m\partial_1 h_{\pmb{\gamma}}(w) & \ldots & \bar{\partial}_m\partial_m h_{\pmb{\gamma}}(w)
      \end{pmatrix}
 = \begin{pmatrix}
       h_{\pmb{\gamma}}(w) & X_{n\times mn}\\
       L_{mn\times n} & S_{mn\times mn}
      \end{pmatrix}.      
\end{eqnarray*}
Computing the $2\times 2$ entry of the inverse of this  block matrix and equating it to $R,$ we have 
\begin{eqnarray*}
 R^{-1} &=& S - L h_{\pmb{\gamma}}(w)^{-1} X \\
        &=& \big(\!\!\big(\bar{\partial}_i\partial_j h_{\pmb{\gamma}}(w)\big)\!\!\big)_{i,j=1}^{m} - \big(\!\!\big((\bar{\partial}_i {h_{\pmb{\gamma}}(w)}) h_{\pmb{\gamma}}(w)^{-1}(\partial_j h_{\pmb{\gamma}}(w))\big)\!\!\big)_{i,j =1}^{m} \\
        &=& \big(\!\!\big(h_{\pmb{\gamma}}(w) \bar{\partial}_i{(h_{\pmb{\gamma}}(w)}^{-1}\partial_j h_{\pmb{\gamma}}(w))\big)\!\!\big) \\
        &=& -\big(\!\!\big(h_{\pmb{\gamma}}(w) \mathcal{K}^{i,j}(\pmb{\gamma})(w)\big)\!\!\big) \\
        &=& -\mathcal{K}_{\pmb{\gamma}} (w),
 \end{eqnarray*}
 where $\mathcal{K}^{i,j}(\pmb{\gamma})(w)$ denote the matrix of the curvature $\mathcal{K}$ at $w\in \Omega^*_0$ w.r.t the frame $\pmb{\gamma}$ of the bundle $E_{\pmb T}$ on $\Omega^*_0$  and $\mathcal{K}_{\pmb{\gamma}(w)} = \big(\!\!\big(\mathcal{K}^{i,j}(\pmb{\gamma})(w)\big)\!\!\big)_{i,j =1}^{m}$. Also, by our choice of the frame $\pmb{\gamma}$ we have $h_{\pmb{\gamma}}(w)=I_n.$  Hence it follows that
 \begin{eqnarray}
 \pmb t(w) \overline{\pmb t(w)}^{\rm tr} = R &=& \big(-\mathcal{K}_{\pmb{\gamma}}(w)\big)^{-1}. \label{eq: RK}
 \end{eqnarray}
 This completes the proof.

\end{proof}


The matrix representation of the operator ${T_i}_{|{\mathcal{M}_w}}$  w.r.t. the orthonormal basis $\mathbf{u}= (\mathbf u_1, \ldots , \mathbf u_{m+1})$ in the subspace $\mathcal{M}_w$ is of the form 
\begin{align*}
\begin{bmatrix}
{T_i}_{|{\mathcal{M}_w}}
\end{bmatrix} _{\mathbf{u}} = \begin{pmatrix}
w_i I_n &  \pmb t_i(w)\\
0_{mn \times n} & w_i I_{mn}
\end{pmatrix},\,\,i=1,\ldots,m.
\end{align*}
Now assuming that the joint spectrum of the tuple $\pmb T$ is contained in $\overbar{\Omega}^*,$ it follows that for any  function $f\in \mathcal O(\overbar{\Omega}^*),$ we have 
\begin{eqnarray*}
f(\pmb T)_{|\mathcal{M}_w} &=& f\big ( \pmb T_{| \mathcal{M}_w} \big )\\
&=&  \begin{pmatrix} f(w) & \bigtriangledown f(w) \cdot  \pmb t(w) \\ 
0 & f(w) \end{pmatrix} = f(\pmb T_w),\\
\end{eqnarray*}
where $\pmb T_w$ is the $m$ tuple of operator $\pmb T_{| \mathcal{M}_w} $ and 
\begin{align*}
\bigtriangledown f(w) \cdot  \pmb t(w) &= \partial_1f(w) \mathbf t_1(w) + \cdots + \partial_m f(w) \mathbf t_m(w)\\
&= \big((\partial_1f(w))I_n,\ldots,(\partial_mf(w)) I_n\big) (\pmb t(w))\\
&= (I_n \otimes \bigtriangledown f(w) ) (\pmb t(w)).
\end{align*}
From equation \eqref{eq: RK}, we also have 
 \begin{align}\label{tw and curvature}
 \pmb t(w) \overline{\pmb t(w)}^{\rm tr}= -(\mathcal{K}_{\pmb{\gamma}}(w))^{-1}.
 \end{align}
As an application, it is easy to obtain a curvature inequality for those commuting tuples of operators $\pmb T$ in the Cowen-Douglas class $B_n(\Omega^*)$ which admit $\overbar{\Omega}^*$ as a spectral set. This is easily done via the holomorphic functional calculus. 

If $\pmb T$ admits $\overbar{\Omega}^*$ as a spectral set, then the inequality $I - f(\pmb T_w)^{*}f(\pmb T_w)\geq 0$ is evident for all holomorphic functions mapping $\overbar{\Omega}^*$ to the unit disc $\mathbb D.$ As is well-known, we may assume without loss of generality that $f(w)=0.$  Consequently, the inequality $I - f(\pmb T_w)^{*}f(\pmb T_w)\geq 0$ with $f(w)=0$ is equivalent to 
\begin{align}\label{Gen CI}
\bigg(\overline{I_n \otimes \bigtriangledown f(w)}^{\rm tr} \bigg) (I_n \otimes \bigtriangledown f(w)) \leq -(\mathcal{K}_{\tilde{\pmb{\gamma}}}(w)).
\end{align}

Let $V\in \mathbb C^{mn}$ be a vector of the form
\begin{align*}
V&=\begin{psmallmatrix}
V_1\\
\cdot\\
\cdot\\
\cdot\\
V_m
\end{psmallmatrix}, 
\text{\,\,where\,\,}
V_i=\begin{psmallmatrix}
V_i(1)\\
\cdot\\
\cdot\\
\cdot\\
V_i(n)
\end{psmallmatrix}\in \mathbb C^n.
\end{align*}
The Carath\'{e}odory norm of the (matricial) tangent vector $V \in \mathbb C^{mn}$ at a point $z$ in $\Omega,$ is defined by 
\begin{align*}
(C_{\Omega,z}(V))^2&= \sup \big\{ \langle \bigg(\overline{I_n \otimes \bigtriangledown f(z)}^{\rm tr} \bigg)(I_n \otimes \bigtriangledown f(z))V ,V \rangle: f\in \mathcal O(\overbar{\Omega}),\|f\|_{\infty} \leq 1, f(z)=0\big \}\\
&= \sup \big\{\sum_{i,j=1}^m\overbar{\partial_if(z)}\partial_jf(z)\langle V_j,V_i\rangle : f\in \mathcal O(\overbar{\Omega}),\|f\|_{\infty} \leq 1, f(z)=0\big \} \\
&= \sup \big\{\|\sum_{j=1}^m \partial_jf(z) V_j\|_{\ell^2}^2 : f\in \mathcal O(\overbar{\Omega}),\|f\|_{\infty} \leq 1, f(z)=0\big \}.
\end{align*}

Now we compute the Carath\'{e}odory norm of the tangent vector $V \in \mathbb C^{mn}$ in the case of Euclidean ball $\mathbb B^m$ and of polydisc $\mathbb D^m.$ For a self map $g = (g_1,g_2,\ldots,g_m): \Omega \rightarrow \Omega$ and 
\begin{align*}
V&=\begin{psmallmatrix}
V_1\\
\cdot\\
\cdot\\
\cdot\\
V_m
\end{psmallmatrix},
\end{align*} let $g_*(z)(V)$ be the vector defined by 
 \begin{align*}
g_*(z)(V)&=\begin{psmallmatrix}
\sum_j \partial _j g_1(z) V_j\\
\cdot\\
\cdot\\
\cdot\\
\sum_j \partial _j g_m(z) V_j
\end{psmallmatrix}.
\end{align*}
From the definition of the Carath\'{e}odory norm, it follows that $C_{\Omega,g(z)}(g_*(z)(V)) \leq C_{\Omega,z}(V).$ In particular we have that $C_{\Omega,\varphi(z)}(\varphi_*(z)(V)) = C_{\Omega,z}(V)$ for any biholomorphic map $\varphi$ of $\Omega.$ The  group of biholomorphic automorphisms of both these domains $\mathbb B^m$ and $\mathbb D^m$ acts transitively. So,  it is enough to compute $C_{\Omega,0}(V),$ since there is an explicit formula relating $C_{\Omega,z}(V)$ to $C_{\Omega,0}(V),$ $\Omega=\mathbb B^m$ or $\mathbb D^m.$

It is well known that the set $\big\{ \triangledown f(0): f\in \mathcal O(\overline{\mathbb B^m}),\|f\|_{\infty} \leq 1, f(z)=0\big \}$ is equal to the Euclidean unit ball  $\mathbb B^m.$ Now for $a=(a_1,a_2,\ldots,a_m)\in \mathbb B^m,$ note that 
\begin{align*}
\|\sum_{j=1}^m a_j V_j\|_{\ell^2}^2= \sum_{i=1}^n |\sum _{j=1}^m a_j V_j(i)|^2 \leq \|a\|_{\ell^2}^2 \sum_{i=1}^n \sum_{j=1}^m |V_j(i)|^2.
\end{align*}
From this it follows that the Carath\'{e}odory norm of the tangent vector $V \in \mathbb C^{mn}$ at the point $0$ in the case of the Euclidean ball $\mathbb B^m$ is equal to the Hilbert-Schmidt norm of $V,$ that is, $ \|V\|^2_{HS}= \sum_{i=1}^n \sum_{j=1}^m |V_j(i)|^2.$ Similarly, in case of polydisc $\mathbb D^m,$ we have 
$\big\{ \triangledown f(0): f\in \mathcal O(\overline{\mathbb D^m}),\|f\|_{\infty} \leq 1, f(z)=0\big \}$ is equal to the $\ell^1$ unit ball of $\mathbb C^m.$ For $a=(a_1,a_2,\ldots,a_m):$  $\|a\|_1 < 1,$  we note that 
\begin{align*}
\|\sum_{j=1}^m a_j V_j\|_{\ell^2} \leq \|a\|_{\ell^1} \max_{j} \|V_j\|_{\ell^2}.
\end{align*}
Thus we conclude  that the Carath\'{e}odory norm of the tangent vector $V \in \mathbb C^{mn}$ at the point $0,$ in the case of the polydisc $\mathbb D^m,$ is equal to $\max\{\|V_j\|_{\ell^2}: 1 \leq j \leq m\}.$ A more detailed discussion on such matricial tangent vectors $V$ and the question ontractivity, complete contractivity of the homomorphism induced by them appears in \cite{Misra93Pati}.

Thus from the definition of the  Carath\'{e}odory norm and  equation \eqref{Gen CI}, a proof of the  theorem below follows.
\begin{thm}\label{Generalized CI}
Let $\pmb T$ be a commuting tuple of operator in $B_n(\Omega)$ admitting $\overbar{\Omega}^*$ as a spectral set. Then for an arbitrary but fixed point $w\in \overbar{\Omega}^*,$ there exist a frame $\pmb{\gamma}$ of the bundle $E_{\pmb T},$ defined in a neighborhood of $w,$ which is orthonormal at $w,$ so that following inequality holds
$$\langle \mathcal{K}_{\pmb{\gamma}}(w)V,V\rangle \leq - (C_{\Omega^*,w}(V))^2\,\,\,\text{for every \,}\,V\in \mathbb C^{mn} .$$
\end{thm}

Now we derive a curvature inequality specializing to the case of a bounded planar domains $\Omega^*.$ 
Using techniques from Sz.-Nagy  Foias model theory for contractions, Uchiyama \cite{Uchiyama90curvatures}, was the first one to prove a curvature inequality for operators in $B_n(\mathbb D).$  To obtain curvature inequalities in the case of finitely connected planar domains $\Omega,$ he considered the contractive operator $F_w(T),$ where $F_w: \Omega \to \mathbb D$ is the Ahlfors map, $F_w(w)=0,$ for some fixed but arbitrary $w\in \Omega.$ The curvature inequality then folllows from the equality $F_w^{\prime}(w)= S_{\Omega}(w,w).$ However, the inequality we obtain below follows directly from the functional calculus applied to the local operators.  More recently, K. Wang and G. Zhang (cf. \cite{Zhang16highcurvature}) have obtained a series of very interesting (higher order) curvature inequalities for operators in $B_n(\Omega).$ 

In the case of bounded finitely connected planar domain with Jordan analytic boundary  the carath\'{e}odory norm of the tangent vector $V \in \mathbb C^{n}$ at a point $z$ in $\Omega$ is given by 
\begin{align*}
(C_{\Omega,z}(V))^2&= \sup \big\{ |f^{\prime}(z)|^2 \langle V,V \rangle_{\ell^2} : f\in \mathcal O(\overbar{\Omega}),\|f\|_{\infty} \leq 1, f(z)=0\big \}\\
&= 4\pi^2 S_{\Omega}(z,z)^2 \langle V,V \rangle_{\ell^2},
\end{align*} 
 (cf. \cite[Theorem 13.1]{BELLcauchy}) where $S_{\Omega}(z,z)$ denotes the S\"{z}ego kernel for the domain $\Omega$ which satisfy
\begin{align*}
2 \pi S_{\Omega}(z,z) = \sup\{|r'(z)|: r \in \mbox{Rat}(\overbar{\Omega}),\norm{r}_{\infty} \leq 1, r(z) =0\}.
\end{align*}
In consequence, we have the following.  
\begin{thm}\label{CI plane domain}
Let $T$ be a operator in $B_n(\Omega^*)$ admitting $\overbar{\Omega}^*$ as a spectral set. Then for an arbitrary but fixed point $w\in {\Omega}^*,$ there exist a frame ${\gamma}$ of the bundle $E_{ T},$ defined on a neighborhood of $w,$ which is orthonormal at $w,$ so that the following inequality holds
$$ \mathcal{K}_{\gamma}(w) \leq - (S_{\Omega^*}(w,w))^2 I_n,\,\,\,\text{for every \,}\,V\in \mathbb C^{n} .$$
\end{thm}

\section{Curvature Inequality and The case of unit disc}

Let $T$ be an operator in  $B_1(\mathbb D)$ and $\mathcal H_K$ be an associated reproducing kernel Hilbert space so that operator $T$ has been realized as $M^*$ on the Hilbert space $\mathcal H_K.$ Without loss of generality we can assume $K_{w}\neq 0$ for every $w\in \mathbb D.$  Let $w_1,\ldots,w_n$ be $n$ arbitrary points in $\mathbb D$ and $c_1,\ldots,c_n$ be arbitrary complex numbers. Using the reproducing property of $K$ and the property that $M^*(K_{w_i})= \bar{w}_iK_{w_i}$ we will have 
\begin{align*}
\|M^*(\sum_{i,j=1}^n c_iK_{w_i})\|^2 = \sum_{i,j=1}^n  w_i \bar{w}_jK(w_i,w_j)c_j\bar{c}_i\,\,,\,\,\|\sum_{i,j=1}^n c_iK_{w_i})\|^2 = (\sum_{i,j=1}^n K(w_i,w_j)c_j\bar{c}_i.
\end{align*}
Let $\tilde{K}(z,w)$  be the function $(1-z\bar{w})K(z,w),$ $z,w\in\mathbb{D}.$ Now it is easy to see that the operator $M^*$ on the Hilbert space $\mathcal H_K$ is a contraction if and only if $\tilde{K}$ is non-negative definite.

\begin{lem}\label{k and dk dependent}
Let $T$ be a contraction in $B_1(\mathbb{D})$ and $\mathcal H_K$ be an associated reproducing kernel Hilbert space. Then for an arbitrary but fixed $\zeta \in \mathbb{D},$ we have $\mathcal{K}_T(\bar{\zeta}) =  -\frac{1}{(1- |\zeta|^2)^2}$ if and only if the vectors $\tilde{K}_\zeta , \bar{\partial}\tilde{K}_{\zeta}$ are linearly dependent in the Hilbert space $\mathcal{H}_{\tilde{K}}.$
\end{lem}
\begin{proof}
Assume $\mathcal{K}_{M^*}(\bar{\zeta}) = -\frac{1}{(1-|\zeta|^2)^2}$ for some $\zeta \in \mathbb{D}.$  Contractivity of $M^*$ gives us the function $\tilde{K}: \mathbb{D}\times \mathbb{D}\mapsto \mathbb{C}$ defined by
\begin{align*}
\tilde{K}(z,w) &= (1-z\bar{w})K(z,w)\;\;\;\;z,w\in\mathbb{D}
\end{align*}
is a non negative definite kernel function. Consequently there exist a  reproducing kernel Hilbert space $\tilde{\mathcal{H}},$ consisting  of complex valued function on $\mathbb{D}$ such that $\tilde{K}$ becomes the reproducing kernel for $\tilde{\mathcal{H}}.$ Also note that $\tilde{K}(z,z)= (1-|z|^2)K(z,z) \neq 0,\;$ for $z\in \mathbb{D}$ which gives us $\tilde{K}_z\neq 0.$  Let $\zeta$ be an arbitrary but fixed point in $\mathbb{D}.$ Now, it is straightforward to verify that $\mathcal{K}_T(\bar{\zeta}) = - \frac{1}{(1-|\zeta|^2)^2}$ if and only if  $\frac{\partial ^2}{\partial{z}\bar{\partial}{z}}\log \tilde{K}(z,z)|_{z=\zeta} = 0.$ Since we have
$$
\frac{\partial ^2}{\partial{z}\bar{\partial}{z}}\log \tilde{K}(z,z)|_{z=\zeta} = -\tfrac{\|\tilde{K}_{\zeta}\|^2 \|\bar{\partial}\tilde{K}_{\zeta}\|^2 - |{\langle \tilde{K}_{\zeta}, \bar{\partial}\tilde{K}_{\zeta} \rangle}|^2}{(\tilde{K}({\zeta},{\zeta}))^2},
$$
Using Cauchy-Schwarz inequality, we see that the proof is complete.  
\end{proof}
\begin{rem}
Two non-zero linear functional $g_1,g_2$ on a vector space are linearly dependent if and only if $\ker (g_1) = \ker( g_2).$ Since $\tilde{K}_{\zeta} \neq 0,$ there are two different possibilities for the linear dependence of the two vectors  $\tilde{K}_{\zeta} , \bar{\partial}\tilde{K}_{\zeta}.$  First,  $\bar{\partial}\tilde{K}_{\zeta} \equiv 0,$ that is, $f^{\prime}(\zeta)=\langle f,\bar{\partial}\tilde{K}_{\zeta}\rangle= 0$ for all $f\in \tilde{\mathcal H}.$ Second,  $\ker \bar{\partial}\tilde{K}_{\zeta}= \ker \tilde{K}_\zeta, $ that is,  the set $\{f \in \tilde{\mathcal{H}}\mid f ^{\prime} (\zeta) =0\}$ is equal to the set $ \{f \in \tilde{\mathcal{H}}\mid f(\zeta) =0\}$
\end{rem}
\begin{rem}\label{criterion for equality at one point}
Let $e(w) = \frac{1}{\sqrt{2}}(\tilde{K}_w \otimes \bar{\partial}\tilde{K}_w - \bar{\partial}\tilde{K}_w \otimes \tilde{K}_w)$ for $w \in \mathbb{D}.$ A straightforward computation shows that $ \|e(w)\|_{\tilde{\mathcal{H}}\otimes \tilde{\mathcal{H}}}^2 = \tilde{K}(w,w)^2 \frac{\partial ^2}{\partial{z}\bar{\partial}{z}}\log \tilde{K}(z,z)|_{z=w}.$ Now if we define 
$$F_K(z,w) := \langle e(z) , e(w)\rangle _{\tilde{\mathcal{H}}\otimes \tilde{\mathcal{H}}}\,\,\text{for}\,\, z,w\in \mathbb{D},$$ 
then clearly $F_K$ is a non negative definite kernel function on $\mathbb{D}\times\mathbb{D}.$ In view of this,  we conclude that $\mathcal{K}_T(\bar{\zeta}) = -(1- |\zeta|^2)^{-2}$ if and only if $F_K(\zeta,\zeta) = 0.$ 
\end{rem}

\begin{prop}\label{hyponormal and equality}
Let $T$ be any contractive co-hyponormal unilateral backward weighted shift operator in  $B_1(\mathbb D).$  If $\mathcal{K}_T(w_0) = -(1-|w_0|^2)^{-2}$ for some $w_0 \in \mathbb{D},$ then the operator  $T$ is unitarily equivalent to $U_+^*,$ the  backward shift operator.
\end{prop}
\begin{proof}
Let $T$ be a contraction in $B_1(\mathbb D)$ and $\mathcal H_K$ be the associated reproducing kernel Hilbert space so that  $T$ is unitarily equivalent to the operator $M^*$ on $\mathcal{H}_K.$ By our hypothesis on $T$ we have that operator $M$ on $\mathcal{H}_K$ is a unilateral forward weighted shift. Without loss of generality, we may assume that the reproducing kernel $K$ is of the  form 
\begin{align*}
K(z,w)= \sum\limits_{n=0}^{\infty}a_n z^n\bar{w}^n,\;  \;z,w \in \mathbb{D}; \;\mbox{where}\; a_n>0\;\; \mbox{for all}\;\; n\geq 0.
\end{align*}
By our hypothesis on the operator $T,$ we have that the operator $M$ on $\mathcal{H}_{K}$ is a contraction. So, the function $\tilde{K}$ defined by $\tilde{K}(z,w)= (1-z\bar{w})K(z,w)$ is a non negative definite kernel function. Consequently, following the Remark \ref{criterion for equality at one point}, the function $F_K (w,w)$ defined by $F_K(w,w)= \tilde{K}(w,w)^2 \tfrac{\partial ^2}{\partial{z}\bar{\partial}{z}}\log \tilde{K}(z,z)|_{z=w}$ is also non negative definite. The  kernel  $K(w,w)$ is a weighted sum of monomials $z^k\bar{w}^k,$ $k=0,1,2, \ldots.$  Hence both $\tilde{K}(w,w)$ and $F_K(w,w)$ are also weighted sums of the same form. So, we have 
\begin{align*}
F_K(w,w) = \sum\limits_{n=0}^{\infty}c_n |w|^{2n}, 
\end{align*} 
for some $c_n \geq 0.$ Now assume $\mathcal{K}_T(\bar{\zeta}) = -\frac{1}{(1- |\zeta|^2)^2}$ for some $\zeta$ in $\mathbb{D}.$  

{\bf{Case 1:}} If $\zeta\neq 0,$ 
then following Remark \ref{criterion for equality at one point}, we have
$$
F_K(\zeta,\zeta) =  \sum\limits_{n=0}^{\infty}c_n |\zeta|^{2n} = 0.
$$
Thus $c_n = 0$ for all $n\geq 0$ since $c_n\geq 0$ and $|\zeta|\neq 0.$
It follows that $F_K$ is identically zero on $\mathbb D\times \mathbb D,$  
that is, $\frac{\partial ^2}{\partial{z}\bar{\partial}{z}}\log \tilde{K}(z,z)|_{z={\bar{w}}} = 0$ for all  $w \in \mathbb{D}.$ Hence 
$$
\frac{\partial ^2}{\partial{z}\bar{\partial}{z}}\log K(z,z)|_{z={\bar{w}}} = \frac{\partial ^2}{\partial{z}\bar{\partial}{z}}\log S_{\mathbb{D}}(z,z)|_{z=\bar{w}}\; \mbox{for all}\; w \in \mathbb{D}. 
$$
Therefore, $\mathcal{K}_T(\bar{w}) = \mathcal{K}_{U_+^*}({\bar{w}})\; \mbox{for all}\; w \in \mathbb{D}$ making  $T \cong  U_+^*.$

Now let's discuss the remaining case, that is $\mathcal{K}_T(\bar{\zeta}) = -\frac{1}{(1- |\zeta|^2)^2},$ for $\zeta =0 \in \mathbb{D}$

{\bf{Case 2:}} If $\zeta = 0,$ then by Lemma \ref{k and dk dependent}, we have $\tilde{K}_0,\bar{\partial}\tilde{K}_0$ are linearly dependent. Now,  \begin{align*}
\tilde{K}(z,w):=(1-z\bar{w})K(z,w) = \sum\limits_{n=0}^{\infty}b_n z^n \bar{w}^n, 
\end{align*} where $b_0= a_0$ and $b_n = a_n - a_{n-1} \geq 0,$ for all $n \geq 1.$ Consequently, we have $\tilde{K}_0(z) \equiv b_0$ and $\bar{\partial}\tilde{K}_0 (z) = b_1z.$ Now $\tilde{K}_0,\bar{\partial}\tilde{K}_0$ are linearly dependent if and only if $b_1=0$ that is $a_0=a_1.$ 


Since $\{\sqrt{a_n} z^n\}_{n=0}^{\infty}$ is an orthonormal basis for the Hilbert space $\mathcal{H}_K,$ the operator $M$ on $\mathcal{H}_K$ is an unilateral forward weighted shift with weight sequence $w_n = \sqrt{\frac{a_n}{a_{n+1}}}$ for $n\geq 0$. So  the curvature of $M ^*$ at the point zero equal to $-1$ if and only if $w_0 = \sqrt{\frac{a_0}{a_{1}}} =1$. Now if we further assume $M$ is hyponormal, that is,  $M^* M \geq MM^*,$ then  the sequence $w_n$ must be increasing. Also contractivity  of $M$ implies that $w_n \leq 1$.  Therefore if $\mathcal{K}_{M ^*} (0) = -1$ for some contractive hyponormal backward weighted shift $M^*$ in $B_1(\mathbb D),$  then it follows that  $w_n =1$ for all $n\geq 1$. Thus any such operator is unitarily equivalent to the backward unilateral shift $U_+^*$  completing the proof of our claim. 
\end{proof}

The proof of {\bf Case 1} given above, actually proves  a little more  than what is stated in the proposition, which we record below as a separate  Lemma. 
 \begin{lem}\label{weighted shift non zero}
 Let $T$ be any contractive unilateral backward weighted shift  operator in  $B_1(\mathbb D).$  If $\mathcal{K}_T(w_0) = -(1- |w_0|^2)^{-2}$ for some  $w_0 \in \mathbb{D}$, $w_0 \not= 0,$ then the operator $T$ is unitarily equivalent to $U_+^*$, the backward shift operator.
 \end{lem} 

Let $T$ be a contraction in $B_1(\mathbb D).$  Let $a$ be a fixed but arbitrary point in $\mathbb D$ and $\phi_{a}$ be an automorphism of the unit disc taking $a$ to $0.$ So, we have $\phi_{a}(z)= \beta(z-a)(1 -\bar{a} z)^{-1}$ for some unimodular constant $\beta.$ Note that
\begin{align*} 
\phi_{a}(T)-w = \beta(T-a)(1-\bar{a} T)^{-1}- w=\beta (T-\tfrac{\bar{\beta}w+a}{1+\bar{a} \bar{\beta}w}) (1+\bar{a}\bar{\beta} w)(1- \bar{a} T)^{-1},\,\,w\in \mathbb D.
\end{align*} 
Since $T$ is in $B_1(\mathbb D),$ we have $(T-wI)$ is fredholm operator for every $w\in\mathbb D.$ Consequently,    it follows that $\phi_{a}(T)-w$ is a fredholm operator and $\phi_{a}(T)$ lies in $B_1(\mathbb D).$
Let $\pmb{\gamma}(w)$ be a frame for the associated bundle $E_T$ of $T$ so that $T(\gamma(w))= w \gamma(w)$ for all $w\in \mathbb D.$ Now it is easy to see that $\phi_{a}(T)(\gamma(w))= \phi_{a}(w) \gamma(w)$ or equivalently $\phi_{a}(T)(\gamma\circ \phi_{a}^{-1}(w))= w (\gamma\circ \phi_{a}^{-1}(w)).$ So, $\gamma\circ \phi_{a}^{-1}(w)$ is a frame for the bundle $E_{\phi_{a}(T)}$ associated with $\phi_{a}(T).$ Hence the curvature $\mathcal K_{\phi_{a}(T)}(w)$ is equal to

\begin{align*}
\frac{\partial ^2}{\partial w \partial \bar{w}}\mbox{log} \|\gamma\circ \phi_{a}^{-1}(w)\|^2= |{\phi_{a}^{-1}}^{\prime}(w)|^2\frac{\partial ^2}{\partial z \partial \bar{z}}\mbox{log}\|\gamma(z)\|^2_{|{z=\phi_{a}^{-1}(w)}}= |{\phi_{a}^{-1}}^{\prime}(w)|^2 \mathcal K_T(\phi_{a}^{-1}(w)).
\end{align*} This leads us to the following transformation rule for the curvature 
\begin{equation} \label{transruleK}
\mathcal K_{\phi_{a}(T)}(\phi_{a}(z)) = \mathcal K_T(z) |\phi_{a}^\prime (z)|^{-2},\,\,\,z\in \mathbb D.
\end{equation} 
Since $|\phi_{a}^\prime(a)|=(1-|a|^2)^{-1},$ in particular we have that
\begin{align}\label{curvature trans rule}
\mathcal K_{\phi_{a}(T)}(0)&= \mathcal K_T(a) (1-|a|^2)^{2}.
\end{align}

{\sf Normalized kernel:}
Let $T$ be an operator in $B_1(\Omega^*)$ and $T$ has been realized as $M^*$ on a Reproducing kernel Hilbert Space $\mathcal H_K$ with non degenerate kernel function $K.$ For any fixed but arbitrary $\zeta\in \Omega,$ the function $K(z,\zeta)$ is non-zero in some neighborhood, say $U,$ of $\zeta.$ 
The function $\varphi_\zeta(z):=K(z,\zeta)^{-1}K(\zeta,\zeta)^{1/2}$ is then holomorphic.  The linear space $(\mathcal H, K_{(\zeta)}):=\{ \varphi_\zeta f: f\in \mathcal H_K\}$ then can be equipped with an inner product making the multiplication operator $M_{\varphi_\zeta}$ unitary. It then follows that 
$(\mathcal H, K_{(\zeta)})$ is a space of holomorphic functions defined on $U \subseteq \Omega,$ it has a reproducing kernel $K_{(\zeta)}$ defined by
\begin{align*}
K_{(\zeta)}(z,w)= K(\zeta,\zeta)K(z,\zeta)^{-1}K(z,w)\overbar{K(w,\zeta)}^{-1},\,\,z,w\in U,
 \end{align*}
 with the property 
 $K_{(\zeta)}(z,\zeta) = 1, z\in U$ and finally the multiplication operator $M$ on $\mathcal H_K$ is unitarily equivalent to the multiplication operator $M$ on $(\mathcal H, K_{(\zeta)}).$ The kernel $K_{(\zeta)}$ is said to be normalized at $\zeta.$

The realization of an operator $T$ in $B_1(\Omega^*)$ as the adjoint of the  multiplication operator on $\mathcal H_K$ is not canonical. However, the kernel function $K$ is determined  upto conjugation by a holomorphic function. Consequently, one sees that the curvature $\mathcal K_K$ is unambiguously defined.  On the other hand, 
Curto and Salinas (cf. \cite[Remarks 4.7 (b)]{Curtosalinas}) prove that the multiplication operators $M$ 
on  two Hilbert spaces $(\mathcal H, K_{(\zeta)})$ and $(\hat{\mathcal H}, \hat{K}_{(\zeta)})$ are unitarily equivalent if and only if $K_{(\zeta)} = \hat{K}_{(\zeta)}$ in some small neighbourhood of $\zeta.$ Thus the normalized kernel at $\zeta,$ that is, $K_{(\zeta)}$ is also unambiguously defined. It follows that  the curvature and the normalized kernel at $\zeta$ serve equally well as a complete unitary invariant for the operator $T$ in $B_1(\Omega^*).$

Let $T$ be a contraction in $B_1(\mathbb D).$ We assume, without loss of generality, that $T$ is unitarily equivalent to the operator $M^*$ on the reproducing kernel Hilbert space $\mathcal H_{T}(\zeta),$ with the reproducing kernel $K^{T}_{(\zeta)}$ having the property that $K^{T}_{(\zeta)}(z,\zeta)=1$ for all $z$ in a neighborhood of $\zeta.$ Let $\phi$ be an automorphism of unit disc $\mathbb D$ such that $\phi(\bar{\zeta})=0.$ We have already seen $\phi(T)\in B_1(\mathbb D).$ So, there exist a reproducing kernel Hilbert Space $\mathcal H_{\phi(T)}(0)$ associated with the kernel function $K^{\phi(T)}_{(0)}$ so that $\phi(T)$ is unitarily equivalent to the operator $M^*$ on the reproducing kernel Hilbert space $\mathcal H_{\phi(T)}(0),$ with the property that $K^{\phi(T)}_{(0)}(z,0)=1$ for all $z$ in a neighborhood of $0.$
Now in the following lemma we describe the relationship between $\mathcal H_{T}(\zeta)$ and $\mathcal H_{\phi(T)}(0).$

\begin{lem} \label{Normalized kernel of phi(T)}
$K^{\phi(T)}_{(0)}(z,w) = K^{T}_{(\zeta)}( {\phi^*}^{-1} (z),{\phi ^*}^{-1}(w))$ and  $\mathcal H_{\phi(T)}(0)=  \mathcal H_{T}(\zeta) \circ {\phi ^*}^{-1} ,$ where ${\phi^*}$ is the automorphism of $\mathbb D$ given by $\phi^*(z) = \overline{\phi(\bar{z})},\,z\in \mathbb D.$ 
\end{lem}
\begin{proof}
For $\phi \in Aut(\mathbb D),$ we have that the function $L(z,w)= K^{T}_{(\zeta)}( {\phi^*}^{-1} (z),{\phi ^*}^{-1}(w))$ is a positive definite kernel and $f\mapsto f\circ {\phi^*}^{-1}$ is a unitary map between $\mathcal H_{T}(\zeta)$ and $\mathcal H_{L}.$ Consequently, the operator $M$ on $\mathcal H_{L}$ is unitarily equivalent to the operator $\phi^*(M) = M_{\phi^*}$ on $\mathcal H_{T}(\zeta).$ Hence, the operator $M^*$ on $\mathcal H_{L}$ is unitarily equivalent to the operator $ {(\phi^*(M))}^*= \phi(M^*)$ on $\mathcal H_{T}(\zeta).$ Also note that as ${\phi^*}^{-1}(0)= \zeta,$ we have $L(z,0)= 1$ for all $z$ in a small neighborhood of $0$, that is, $L$ is normalized at $0.$ Hence we get that $K^{\phi(T)}_{(0)}(z,w) =L(z,w)$ and the associated Hilbert space $\mathcal H_{\phi(T)}(0)=  \mathcal H_{T}(\zeta) \circ {\phi ^*}^{-1} .$
\end{proof}

Observe that as the kernel function $K^{T}_{(\zeta)}$ is normalized at $\zeta,$ we have $1\in  \mathcal H_{T}(\zeta).$ So we have that $ \vee \{1, \phi^*, {\phi^*}^2,\ldots\} \subset \mathcal H_{T}(\zeta).$ Now, it is straightforward to verify that polynomials are dense in $\mathcal H_{\phi(T)}(0)$ if and only if $ \mathcal H_{T}(\zeta) = \vee \{1, \phi^*, {\phi^*}^2,\ldots\} .$


In the following theorem it is shown that the question of R. G. Douglas has an affirmative answer in the class of operator whose adjoint is a $2$ hyper-contraction, with a mild assumption on the Hilbert space $\mathcal H_{T}(\zeta).$ 

First let's recall the definition of $2$ hyper-contraction (cf. \cite{Aglerhypercontraction}). An operator $A$ acting on a Hilbert space $\mathcal H$ is said to be $2$ hyper-contraction if $I-A^*A \geq 0$ and ${A^*}^2A^2-2A^*A +I\geq 0.$ For example every contractive subnormal operator is a $2$ hyper-contraction (cf. \cite[Theorem 3.1]{Aglerhypercontraction}). The following lemma will be very useful in establishing our next result.
\begin{lem}\label{hypercontraction}
Let $A$ be a $2$ hyper-contraction and $\varphi$ be a bi holomorphic automorphism of unit disc $\mathbb D$. Then $\varphi(A)$ is also a $2$ hyper-contraction.
\end{lem}
\begin{proof}
Let $A$ be a $2$ hyper-contraction. Let $\varphi$ be the automorphism of the unit disc $\mathbb D$ given by $\varphi(z)= \lambda \frac{z-a}{1-\bar{a}z}$ for some unimodular constant $\lambda$ and $a\in \mathbb D.$ So $\varphi(A) = \lambda (A-a)(1-\bar{a}A)^{-1}.$ Since $A$ is a contraction, using von-Neuman's inequality we have $\varphi(A)$ is also a contraction. Thus
\begin{align*}
{\varphi(A)^*}^2 \varphi(A)^2 - 2\varphi(A)^*\varphi( A) + I \hskip 34em\\
= (1- a A^*)^{-2}\bigg\{ (A^*-\bar{a})^2 (A-a)^2 - 2 (1- a A^*)(A^*-\bar{a})(A-a)(1-\bar{a}A) \hskip 16em&\\                          
+(1- a A^*)^{2}(1-\bar{a}A)^2\bigg \}(1-\bar{a}A)^{-2} \hskip 8em\\
=  (1- a A^*)^{-2}\bigg\{ (A^*-\bar{a})^2 (A-a)^2 -  (A^*-\bar{a})(1- a A^*)(1-\bar{a}A)(A-a) \hskip 16em \hskip 1ex\\
- (1- a A^*)(A^*-\bar{a})(A-a)(1-\bar{a}A) + (1- a A^*)^{2}(1-\bar{a}A)^2\bigg \}(1-\bar{a}A)^{-2} \hskip 8em \\
= (1- a A^*)^{-2}\bigg\{(A^*-\bar{a})\{(A^*-\bar{a})(A-a)-(1- a A^*)(1-\bar{a}A)\}(A-a) \hskip 16em\\
 - (1- a A^*)\{(A^*-\bar{a})(A-a)-(1- a A^*)(1-\bar{a}A)\} (1-\bar{a}A)\bigg\}(1-\bar{a}A)^{-2}\hskip 8em\\
=  (1- a A^*)^{-2}\bigg\{(A^*-\bar{a})(A^*A-1)(1-|a|^2)(A-a) \hskip 24em \hskip 2ex\\
 - (1- a A^*)(A^*A-1)(1-|a|^2)(1-\bar{a}A)\bigg \}(1-\bar{a}A)^{-2} \hskip 8em\\
 = (1- a A^*)^{-2}(1-|a|^2)\bigg\{(A^*-\bar{a})(A^*A-1)(A-a) \hskip 24em \hskip 2ex\\
   - (1- a A^*)(A^*A-1)(1-\bar{a}A)\bigg \}(1-\bar{a}A)^{-2}\hskip 8em\\
 =  (1- a A^*)^{-2}(1-|a|^2)\bigg\{(1-|a|^2)({A^*}^2A^2 - 2A^*A +I)\bigg\}(1-\bar{a}A)^{-2}\hskip 17em \hskip 1.75ex\\
 = (1- a A^*)^{-2}(1-|a|^2)({A^*}^2A^2 - 2A^*A +I) (1-|a|^2)(1-\bar{a}A)^{-2}.\hskip 18em \hskip 2ex
\end{align*}
Since $A$ is a $2$ hyper-contration, it follows that $\varphi(A)$ is also a $2$ hyper-contraction.
\end{proof}

Now we state our main theorem of this section. For a contraction $T$ in $B_1(\mathbb D),$ we have seen that $\phi(T)$ is also in $B_1(\mathbb D)$ for any automorphism $\phi$ of the unit disc $\mathbb D.$ Recall that $\mathcal H_{\phi(T)}(0)$ denotes the reproducing kernel, normalized at the point $0,$  associated to $\phi(T).$  In the following theorem we answer the question of R. G. Douglas in affirmative in the class of operator whose adjoint is a $2$ hyper-contraction, with a assumption of polynomial density in the Hilbert space $\mathcal H_{\phi(T)}(0)$ for every automorphism $\phi$ of the unit disc $\mathbb D.$

\begin{thm}\label{uniqueness in unit disc}
Let $T$ be an operator in $B_1(\mathbb D).$ Assume that $T^*$ is a $2$ hyper-contraction and polynomials are dense in $\mathcal H_{\phi(T)}(0)$ for every automorphism $\phi$ of unit disc $\mathbb D.$ Then for any such operator $T$ in  $B_1(\mathbb D),$ if  
$\mathcal K_T(\zeta)=-(1- |\zeta|^2)^{-2},$ for an arbitrary but fixed point $\zeta \in \mathbb D,$ then $T$ must be unitarily equivalent to $U_+^*,$ the  backward shift operator.
\end{thm}

\begin{proof}
Let $T$ be an operator in $B_1(\mathbb D)$ such that the adjoint $T^*$ is a $2$ hyper-contraction. Let $\phi_{\zeta}$ be an automorphism of unit disc $\mathbb D$ and $P$ be the operator $\phi_{\zeta}(T).$ We have seen that $P$ is in $B_1(\mathbb D)$ and from Lemma \ref{hypercontraction} it follows that the adjoint $P^*$ is a $2$ hyper-contraction. Now assume $\mathcal K_T(\zeta)= -(1-|\zeta|^2)^{-2}.$ Following \eqref{curvature trans rule}, we see that $\mathcal K_{P}(0)=-1.$ 

Let $K$ denotes the kernel function for the reproducing kernel Hilbert space $\mathcal{H}_{\phi_{\zeta}(T)}(0)$ so that $P$ is unitarily equivalent to the operator $M^*$ on the reproducing kernel Hilbert space $\mathcal{H}_{K}( = \mathcal{H}_{\phi_{\zeta}(T)}(0)).$  Note that $K$ is normalized at $0,$ that is,  $K_{0}(z) =K(z,0) =1$ for all $z$ in some neighborhood of $0.$ By our assumption, polynomials are dense in $\mathcal{H}_K.$ Now we claim that $\bar{\partial}K(\cdot,0)= z.$ As $\mathcal H_K$ consists of holomorphic function, for any $f\in \mathcal H_K,$ we have $f(z) = \sum\limits_{j=1}^{\infty} a_j z^j,$ where $$a_j= \frac{f^{(j)}(0)}{j!} = \langle f, \frac{\bar{\partial}^j K(\cdot,0)}{j!}\rangle.$$

Let $V_j = \frac{\bar{\partial}^j K(\cdot,0)}{j!}.$ To prove $V_1=\bar{\partial}K(\cdot,0)= z,$ it is sufficient to show that $\langle V_1,V_j\rangle =0$ for all $j\geq 0,$ except $j=1.$ First note that since $K(z,0)=1= K(0,z),$ we have $\bar{\partial} K(0,0) = 0.$ It follows that $\langle V_1,V_0\rangle =0.$ Since $K$ is normalized at $0,$ we also have $\mathcal K_P(0)= - \partial \bar{\partial} K(0,0) = - \|V_1\|^2.$ Hence we find that $\|V_1\|^2=1.$ Now to show $\langle V_1,V_j\rangle =0$ for $j\geq 1,$ we need the following  lemma.
\begin{lem}\label{gram decrease}
Let $V$ and $W$ be two finite dimensional inner product space and $A:V \to W$ be a linear map. Let $\{v_1,v_2,\ldots,v_k\}$ be a basis for $V$ and $G_{v},$ (resp. $G_{Av}$) be the grammian  $(\!\!(\langle v_j, v_i \rangle_{V} )\!\!)$ (resp.  $(\!\!(\langle Av_j, Av_i \rangle_{W} )\!\!)$).  The linear map $A$ is a contraction if and only if $ G_{Av} \leq G_v.$ 
\end{lem}
\begin{proof}
Let $x= c_1v_1+c_2v_2+\cdots+c_nv_n$ be an arbitrary element in $V.$ Then the easy verification  that $\|Ax\|^2_{W} \leq \|x\|^2_{V}$ is equivalent $\langle G_{Av} c, c\rangle \leq \langle G_{v}c,c \rangle$ completes the proof.  
\end{proof}
As $(M^*-\bar{w})K(\cdot,w)=0,$ it is easily verified  that $(M^*-\bar{w})\frac{\bar{\partial}K^j(\cdot,w))}{j!}= \frac{\bar{\partial}K^{j-1}(\cdot,w))}{(j-1)!}$ for all $j\geq 1.$ So, we have $M^*(V_j) = V_{j-1}$ for $j\geq 1$ and $M^*(V_0) = 0.$ We also have $\|M^*\|\leq 1.$ Now applying the  lemma $\ref{gram decrease}$ to the set of vector $\{V_0,V_1,\ldots,V_n\}$ we get that 
\begin{align*}
\begin{pmatrix}
\langle V_0, V_0\rangle & \langle V_1, V_0\rangle & \cdots  & \langle V_n, V_0\rangle\\
\langle V_0, V_1\rangle & \langle V_1, V_1\rangle & \cdots  & \langle V_n, V_1\rangle\\
\vdots & \vdots & \ddots & \vdots\\
\langle V_0, V_n\rangle & \langle V_1, V_n\rangle & \cdots  & \langle V_n, V_n\rangle
\end{pmatrix} - 
\begin{pmatrix}
0 & 0 & \cdots  & 0\\
0 & \langle V_0, V_0\rangle & \cdots  & \langle V_{n-1}, V_0\rangle\\
\vdots & \vdots & \ddots & \vdots\\
0 & \langle V_0, V_{n-1}\rangle & \cdots  & \langle V_{n-1}, V_{n-1}\rangle
\end{pmatrix} \geq 0.
\end{align*}
Since $\|V_0\|^2 = K(0,0) =1$ and $\|V_1\|^2=1,$  $(2,2)$ entry of left hand side is $0.$ As left hand is a positive semidefinite matrix, this gives that 2nd row and 2nd column must be identically zero ( for positive semidefinite matrix $B,$ $\langle Be_2,e_2\rangle = 0$ gives $\sqrt{B}e_2= 0,$ and hence $Be_2=0.$) Consequently we get that $\langle V_j,V_1\rangle = \langle V_{j-1}, V_0\rangle$ for all $j= 2,\ldots,n.$ But as $K(z,0)=1=K(0,z),$ it follows that $\bar{\partial}^{k}K(0,0)= \langle V_k,V_0\rangle= 0$ for all $k\geq 1.$ Hence we get that $\langle V_j,V_1\rangle = 0$ for all $j \geq 2.$ Hence we get that $V_1= \bar{\partial}K(\cdot,0)=z$ and $\|z\|^2=\|V_1\|^2= 1.$ We also have $V_0= K(\cdot,0)=1$ with $\|1\|^2=\|V_0\|^2 = K(0,0)=1.$
 

%


By our assumption, the operator $M$ on $\mathcal H_K$ is a $2$ hyper-contraction. In particular $M$ is also a contraction and $\|1\|_{\mathcal H_K} =1.$ Hence we have $\|z^n\|_{\mathcal H_K} \leq 1,$ for all $n\geq 1.$ Since $M$ on $\mathcal H_K$ is a $2$ hyper-contraction, that is, $I - 2 M^*M +{M^*}^2M^2 \geq 0,$ equivalently, $\|f\|_{\mathcal H_K} ^2 - 2 \|zf\|_{\mathcal H_K} ^2 + \|z^2 f\|_{\mathcal H_K} ^2 \geq 0,$ for all $f\in \mathcal H_K.$ Since $\|1\| =\|z\| =1,$ taking  $f =1,$ we have $\|z^2\| \geq 1.$ But we also have $\|z^2\| \leq 1,$ which gives us $\|z^2\| = 1.$ Inductively, by choosing $f= z^k,$ we obtain $\|z^{k +2}\| =1$ for every $k \in \mathbb N.$ Hence we see that $\|z^n\| = 1$ for all $n \geq 0.$

We  use  Lemma $\ref{gram decrease}$  to show that $\{z^n\mid n \geq 0\}$ is an orthonormal set in the Hilbert space $\mathcal H_K,$ 
Consider the two subspace $V$ and $W$ of $\mathcal H_K,$ defined by $V = \vee \{1,z,\cdots,z^k\}$  and $W = \vee\{z,z^2,\cdots,z^{k+1}\}.$ Since $M$ is a contraction, applying the lemma we have just proved, it follows that the matrix $B$ defined by  
\begin{align*}
B = \begin{pmatrix}
\langle z^j,z^i\rangle
\end{pmatrix}_{i,j=0}^k - \begin{pmatrix}
\langle z^{j+1},z^{i+1}\rangle
\end{pmatrix}_{i,j=0}^k
\end{align*}is  positive semi-definite. But we have $\|z^i\| =1,$ for all $i\geq 0.$ Consequently, each diagonal entry of  $B$ is zero. Hence $tr(B) =0.$ Since $B$ is positive semi-definite, it follows that $B=0.$ Therefore,   $\langle z^j, z^i \rangle = \langle z^{j+1},z^{i+1}\rangle$ for all $0\leq i,j\leq k.$ We have $K_0(z) \equiv 1.$ So, $M^*1 = M^*(K_0) =0.$ From this it follows that for any $k \geq 1,$ we have $\langle z^k, 1 \rangle = \langle z^{k-1},M^*1 \rangle =0.$ This together with  $\langle z^j, z^i \rangle = \langle z^{j+1},z^{i+1}\rangle$ for all $0\leq i,j\leq k,$ inductively shows that  $\langle z^j, z^i \rangle =0$ for every $i\neq j.$ Hence $\{z^n\mid n \geq 0\}$ forms an orthonormal set.

Since by our hypothesis polynomials are dense in $\mathcal H_K,$ the set of vectors $\{z^n \mid n \geq 0\}$ forms an orthonormal basis for $\mathcal H_K.$ Hence the multiplication operator $M$ on $\mathcal H_K$ is unitarily equivalent to $U_+,$ the unilateral forward shift operator. Consequently  $P$ is unitarily equivalent to $U_+^*.$ But $U_{+}^*$ being a homogeneous operator, we have $U_{+}^*$ is unitarily equivalent to $\phi_{\zeta}^{-1}(U_{+}^*)$ (cf. \cite{GMback}). Hence, we infer that $T= \phi_{\zeta}^{-1}(P)$ is unitarily equivalent to $U_+^*.$ 
\end{proof}
\section{Curvature Inequality  in the case of finitely connected domain}

Let $dv$ be the Lebesgue  area measure in the complex plane $\mathbb C$ and $h$ be a positive continuous function on $\Omega$ which is integrable w.r.t the area measure $dv.$ Now  consider the weighted Bergman space $(\mathbb A^2(\Omega), h dv)$ consists of  all holomorphic function $f$ on $\Omega$ satisfying $\|f\|_h^2 = \int_{\Omega}|f(z)|^2 h(z)dv(z) < \infty.$ Since for any compact set $C \subset \Omega,$ the function $h$ being bounded below on $C,$ it follows that evaluation at a point in $\Omega$ is a locally uniformly bounded linear map on $(\mathbb A^2(\Omega), h dv)$ and consequently  $(\mathbb A^2(\Omega), h dv)$ is a reproducing kernel Hilbert space. Let $K(z,w)$ be the kernel function for $(\mathbb A^2(\Omega),  hdv).$ It is well known that the multiplication operator $M$ by co-ordinate function on $(\mathbb A^2(\Omega), h dv)$  is a subnormal operator having $\overbar{\Omega}$ as a spectral set.

Let $w$ be an arbitrary but fixed point in $\Omega.$  Let $\mathcal M_w$ be the closed convex set in $\mathcal H = (\mathbb A^2(\Omega),  hdv)$ defined by $\mathcal M_w = \{ f\in \mathcal H : f(w) =0, f^{\prime}(w)= 1\}.$ Consider the following extremal problem 
\begin{align*}
 \inf \{\|f\|^2: f \in \mathcal M_w \}.
\end{align*}
Let $\mathcal E_w$ be the subspace of $\mathcal H$ defined by 
\begin{align*}
\mathcal E_w = \{ f\in \mathcal H : f(w) =0, f^{\prime}(w)= 0\}. 
\end{align*}
Since $f+g \in \mathcal M_w,$ whenever $f\in \mathcal M_w$ and $g \in \mathcal{E}_{w},$ It is evident that the unique function $F$ which solves the extremal problem must belong to $\mathcal{E}_{w}^{\perp}.$ From the reproducing property of $K,$ it follows that 
\begin{align*}
f(w) = \langle f, K(\cdot, w) \rangle, \,\,f^{\prime}(w)= \langle f, \bar{\partial}K(\cdot,w) \rangle .
\end{align*} Consequently, we have $\mathcal{E}_{w}^{\perp} = \vee \{ K(\cdot, w), \bar{\partial}K(\cdot,w) \}. $ 
A solution to the extremal problem  mentioned  above can be found in terms of the kernel function as in \cite{GMCI}:
\begin{align*}
\inf\; \{\|f\|^2:f\in M_w\} &= \bigg\{K(w,w) \bigg(\frac{\partial ^2}{\partial z \partial \bar{z}}\mbox{log} K(z,z)|_{z=w}\bigg)  \bigg\}^{-1}.
\end{align*}
Now consider the function $g$ in $\mathcal H$ defined by
\begin{align*}
g(z):= \frac{ K_{w}(z) F_{w}(z)}{2 \pi S(w,w) K(w,w)}, \;\;\;z\in \Omega,
\end{align*} 
where $F_{w}(z) = \frac{S_{w}(z)}{L_{w}(z)}$ denotes the Ahlfors map for the domain $\Omega$ at the point $w$ (cf. \cite[Theorem 13.1]{BELLcauchy}). Note that $|F_w(z)|<1$ on $\Omega$ and $|F_{w} (z)| \equiv 1$ on $\partial \Omega.$
As $g\in \mathcal H,$ we have the inequality  
\begin{align*}
\bigg\{K(w,w) \bigg(\frac{\partial ^2}{\partial z \partial \bar{z}}\mbox{log} K(z,z)|_{z=w}\bigg)  \bigg\}^{-1} & \leq \|g\|^2\\
 &= \frac{1}{4 \pi ^2 S(w,w)^2 K(w,w)^2}\int_{\Omega} |F_{w}(z)|^2 |K(z,w)|^2  h(z) dv(z) \\
 &< \frac{1}{4 \pi ^2 S(w,w)^2 K(w,w)^2}\int_{\Omega} |K(z,w)|^2 h(z) dv(z), \\
 &= \frac{1}{4 \pi ^2 S(w,w)^2 K(w,w)},
\end{align*}
where the last but one  strict inequality follows from the inequality $|F_w(z)|<1$ on $\Omega.$ 
Hence we have ${\partial_z \bar{ \partial}}_z \mbox{log} K(z,z)|_{z=w} > 4 \pi ^2 S(w,w)^2,$ which is the strict curvature inequality. 
This together with Theorem 2.6 of \cite{MRuniqueness} proves the following uniqueness theorem.  

Let 
$\mathfrak h=\{h: h \text{ is a positive continuous integrable (w.r.t area measure) function on}\,\, \Omega \}$  
and similarly let 
$
\hat{\mathfrak h}= \{\hat{h}: \hat{h} \text{ is a  positive continuous function on}\,\, \partial\Omega\}.
$
Finally, let $\mathcal F_1=\{ M\, \text{on}\, (\mathbb A^2(\Omega),  hdv): h \in \mathfrak h\}$ and $\mathcal F_2= \{ M\, \text{on}\, (H^2(\Omega),  \hat{h} ds):\hat{h} \in \hat{\mathfrak h}\}.$ Set $\mathcal F = \mathcal F_1 \cup \mathcal F_2.$  
\begin{thm}\label{uniqueness in finitely connected}
 Let $\zeta$ be an  arbitrary but fixed point in $\Omega$ and  $T$ be an  operator in $B_1(\Omega^*).$  Assume that the adjoint $T^*$ is in $\mathcal F.$ Then 
$\mathcal K_T(\bar{\zeta}) \leq - 4\pi^2 S_{\Omega}(\zeta,\zeta)^2,$ equality occurs for a unique operator, upto unitary equivalence.
\end{thm}

It was shown in \cite{MRuniqueness} that the class of operators in $\mathcal  F_2$  include the bundle shifts introduced in \cite{ABsubnormal}. We conclude this section by showing that the class $\mathcal F_1$ includes the class of Bergman bundle shift of rank $1$ introduced in \cite{Bergmanbundle}. Let $\mathcal G$ be the class of operators  contained in $\mathcal F$ defined by 
$ \mathcal G = \{M\, \text{on}\, (\mathbb A^2(\Omega),  hdv)$:  $\log h$ is harmonic on $\overbar{\Omega} \}.$ After recalling the definition of of Bergman bundle shift (cf. \cite{Bergmanbundle}), we proceed to establish the existence of a surjective map from $\mathcal G$ onto the class of Bergman bundle shift of rank $1$.

Let $\pi:\mathbb{D} \to \Omega$ be a holomorphic covering map. Bergman bundle shifts is realized as a multiplication operator on a certain subspace of the weighted Bergman space $(\mathbb A^2(\mathbb{D}), |\pi^{\prime}(z)|^2dv(z)).$  Let $G$ denote the group of deck transformation associated to the map $\pi$ that is $G = \{A \in \rm Aut(\mathbb{D}) \mid \pi \circ A = \pi \}.$ 
Let $\alpha$ be a character, that is, $\alpha \in \text{Hom}(G, \mathbb S^1).$ A  holomorphic function $f$ on unit disc $\mathbb{D}$ satisfying $f \circ A = \alpha (A) f , \;\;\text{for all}\; A \in G$, is called a modulus automorphic function of index $\alpha.$ Now consider the following subspace of the weighted Bergman space $(\mathbb A^2(\mathbb{D}), |\pi^{\prime}(z)|^2dv(z))$ which consists of modulus automorphic function of index $\alpha$, namely  
\begin{align*}
\mathbb A^2(\mathbb{D},\alpha) = \{ f \in (\mathbb A^2(\mathbb{D}), |\pi ^{\prime}(z)|^2 dv(z) ) \mid f \circ A = \alpha (A) f , \;\;\text{for all}\; A \in G\}
\end{align*}
 Let $T_{\alpha}$ be the multiplication operator by the covering map $\pi$ on the subspace $\mathbb A^2(\mathbb{D},\alpha).$ The operator $T_{\alpha}$ is called a Bergman bundle shift of rank 1 associated to the character $\alpha.$
 
Like the Hardy bundle shift (cf. \cite{ABsubnormal}) there is another way to realize the Bergman bundle shift as a multiplication operator $M$ on a Hilbert space of multivalued holomorphic function defined on  $\Omega$ with the property that its absolute value is single valued. A multivalued holomorphic  function defined on  $\Omega$ with the property that its absolute value is single valued is called a multiplicative function. Every modulus automorphic function $f$ on $\mathbb{D}$ induce a multiplicative function on $\Omega,$ namely, $f \circ \pi ^{-1}.$ Converse is also true (cf. \cite [Lemma 3.6]{VOICHICKideal}). We define the class $\mathbb A^2_{\alpha}(\Omega)$  consisting of multiplicative function in the following way: 
\begin{align*}
\mathbb A^2_{\alpha}(\Omega):=\{f\circ \pi^{-1}\mid f\in \mathbb A^2(\mathbb{D},\alpha) \}
\end{align*}

So the linear space $\mathbb A^2_{\alpha}(\Omega)$  consists of those multiple valued function $h$ on $\Omega$ for which $|h|$ is single valued, $|h|^2$ is integrable w.r.t area measure $dv$ on $\omega$ and $h$ is locally holomorphic in the sense that each point $w \in \Omega$ has a neighborhood $U_w$ and a single valued holomorphic function $g_{w}$ on $U_w$ with the property $|g_{w}| = |h|$ on $U_w$ (cf. \cite[p.101]{FISHER}). It follows that  the linear space $\mathbb A^2_{\alpha}(\Omega)$ endowed with the norm
\begin{align*}
\|f\|^2 &= \int _{\Omega} |f(z)|^2 dv(z),                        
\end{align*} is a Hilbert space. We denote it by $\big(\mathbb A^2_{\alpha}(\Omega),dv \big).$ In fact the map $f \mapsto f \circ \pi ^{-1}$ is a unitary map from $\mathbb A^2(\mathbb{D},\alpha)$ onto $\big(\mathbb A^2_{\alpha}(\Omega),dv\big)$ which intertwine the multiplication by $\pi$ on $\mathbb A^2(\mathbb{D},\alpha)$ and the multiplication by coordinate function $M$ on $\big(\mathbb A^2_{\alpha}(\Omega),dv\big).$ Thus the multiplication operator $M$ on $\big(\mathbb A^2_{\alpha}(\Omega),dv\big)$ is also called Bergman bundle shift of rank 1.

Let $h$ be a positive function on $\overbar{\Omega}$ with $\log h$ harmonic on $\overbar{\Omega}.$ 
Now we show that the the multiplication operator $M$ on the weighted Bergman space $(\mathbb A^2(\Omega),  hdv)$ is unitarily equivalent to a Bergman bundle shift $T_{\alpha}$ for some character $\alpha.$  
In this realization, it is not hard to see that all the Bergman bundle shift of rank 1 are in the same similarity class. First note that as $h$ is both bounded above and below. So, there exist positive constants $p,q$ such that $0 < p\leq h(z)\leq q$ for all $z\in \overbar{\Omega}.$ Consequently, we have 
\begin{align*}
p\|\cdot\|_{1} \leq \|\cdot\|_{h} \leq q \|\cdot\|_{1}.
\end{align*}
Thus the norms on weighted Bergman space $(\mathbb A^2(\Omega),  h\,dv)$ is equivalent to the norm on the Bergman space $(\mathbb A^2(\Omega),  dv).$ It follows that the identity map is an invertible operator between these two Hilbert spaces and intertwines the associated multiplication operator. This shows that every operator in the class $\mathcal G$ is similar to the multiplication operator $M$ on the Bergman space $(\mathbb A^2(\Omega),  dv).$

The following lemma is the essential step in proving the existence of a bijective map from $\mathcal G$ to the class of Bergman bundle shift of rank 1.

\begin{lem}\label{existence of modulus auto}
Let $h$ be a positive function on $\overbar{\Omega}$ such that $\log h$ is harmonic on $\overbar{\Omega},$ then there exist a  function $F$ in $H^{\infty}_{\gamma}(\Omega)$ for some character $\gamma$ such that $|F|^2 = h$ on ${\Omega}.$ In fact $F$ is invertible in the sense that there exist $G$ in $H^{\infty}_{\gamma ^{-1}}(\Omega)$ so that $FG = 1$  on $\Omega.$ Furthermore, given any character $\gamma$ there exists a positive function $h$ on $\overbar{\Omega}$ such that $\log h$ is harmonic on $\overbar{\Omega}$ and $h= |F|^2$ on $\Omega$ for some  $F$ in $H^{\infty}_{\gamma}(\Omega).$ 
\end{lem}
\begin{proof}
 Let $u^*$ be the multiple value conjugate harmonic function of $\frac{1}{2}\log h$ defined on $\Omega.$  Let's denote the period of the multiple valued conjugate harmonic function $u^*$ around the boundary component $\partial \Omega_{j}$ by 
\begin{align*}
   c_j = -\int _{\partial \Omega _{j}} \frac{\partial}{\partial \eta _z}\big(\frac{1}{2}\log h(z)\big)ds_z, \;\;\;\mbox{for}\; j= 1,2,...,n.\;\;
\end{align*}
The negative sign in this equation appears since we have assumed that $\partial \Omega$ is positively oriented, hence the different components of the boundary $\partial \Omega_j,$ $j=1,2,\ldots, n,$ except the outer one are  oriented in clockwise direction. Now consider the function $F(z)$ defined by 
\begin{align*}
F(z) &= \exp(\frac{1}{2}\log h(z) + i u^*(z)). 
\end{align*} 
Now observe that $F$ is a multiplicative holomorphic  function  on $\Omega.$ Hence following \cite[Lemma 3.6]{VOICHICKideal}, we have a existence of modulus automorphic function $f$ on unit disc $\mathbb{D}$ so that $F = f \circ \pi ^{-1}.$ We find the index of the modulus automorphy for the function $f$ in the following way. Around each boundary component $\partial \Omega_{j},$ along the anticlockwise direction, the value of $F$ gets changed by $\exp(i c_j)$ times its initial value. So, the index of $f$ is determined by the $n$ tuple of numbers $(\gamma_1, \gamma_2,...,\gamma_n)$ given by,
\begin{align*}
\gamma_j &= \exp (i c_j),  \;\;\;\; j= 1,2,...,n.\;\;
\end{align*}
For each of these $n$ tuple of numbers, there exist a homomorphism $\gamma :\pi_1(\Omega) \to \mathbb{T}$ such that these $n$ tuple of numbers occur as  a image of the $n$ generator of the group $\pi_1(\Omega)$ under the map $\gamma.$ Hence $F$ belongs to $H^{\infty}_{\gamma}(\Omega)$ with $|F|^2 = h$ on $ \Omega.$ Note that  $|F(z)|$ is both bounded above and below on $\Omega.$ 

The function $\frac{1}{h}$ is also positive and $\log \frac{1}{h}$ is harmonic, as before, there exists a function $G$ in $H^{\infty}_{\delta}(\Omega)$ with $|G|^2 =\frac{1}{h}$ on $\partial \Omega.$ Since $\log\frac{1}{h} = - \log h$, it follows that index of $G$ is exactly $(\gamma_1^{-1}, \gamma_2^{-1},...,\gamma_n^{-1})$ and hence $\delta$ is equal to $\gamma^{-1}.$ Evidently $FG =1$ on $\Omega.$

For the last part of the lemma, recall that there exist functions $\omega _j(z)$ which is harmonic in $\Omega$ and has the boundary values $1$ on $\partial\Omega_j$  and is $0$ on all the other boundary components, for each $j=1,2,\ldots,n.$ Since the boundary of $\Omega$ consists of Jordan analytic curves, we have that the functions $\omega_j(z)$ is also harmonic on $\overbar{\Omega}.$ Let $p_{i,j}$ be the periods of the harmonic function $\omega_j$ around the boundary component $\partial\Omega_i,$ that is,

\begin{align*}
p_{i,j} &=  -\int _{\partial \Omega _{i}} \frac{\partial}{\partial \eta _z}\big(\omega_j(z)\big)ds_z, \;\;\;\mbox{for}\; i, j= 1,2,...,n\;\;
\end{align*}
The negative sign appears in the equation as it is assumed that $\partial \Omega$ is positively oriented, that is,  the boundary  components $\partial \Omega_j,$ $j=1,2,\ldots, n,$ except the outer one, namely $\partial \Omega_{n+1},$ are oriented in clockwise direction. So the period of the harmonic function $u(z)= a_1\omega_1(z)+a_2 \omega_2(z)+\cdots+a_n \omega_n(z)$ around the boundary component $\partial\Omega_i,$ is equal to $\sum_j p_{i,j}\alpha_j.$ It is well known that the $n\times n$ period matrix $(\!\!( p_{i,j} )\!\!)$ is positive definite and hence invertible (cf. \cite[Sec.10,Ch 1]{Nehari75conformal}). Thus it follows that for any $n$- tuple of real number, say $(b_1,b_2,\ldots,b_n)$ we have a harmonic function $u$ on $\overbar{\Omega}$ such that its period around boundary component $\partial\Omega_i,$ is equal to $b_i.$ Let $g$ be the positive function on $\overbar{\Omega}$ defined by $g(z)= \exp (2 u(z)),\,z\in\overbar{\Omega}.$ Now following the first part of the lemma, we have that there exists a $F$ in $H^{\infty}_{\gamma}(\Omega)$ such that $|F|^2 = g$ on $\overbar{\Omega}.$ Furthermore the character $\gamma$ is determined by 
\begin{align*}
\gamma_j = \exp(ib_j),\,\,\text{for}\, j=1,2,\ldots,n.
\end{align*}
As this is true for arbitrary $n$- tuple of real number $(b_1,b_2,\ldots,b_n),$ the result follows. 
\end{proof}

Let $h$ be a positive function on $\overbar{\Omega}$ such that $\log h$ is harmonic on $\overbar{\Omega}.$ We see that there is a $F$ in $H^{\infty}_{\gamma}(\Omega)$ with $|F|^2 = h$ on $\Omega$ and a $G$ in $H^{\infty}_{\gamma ^{-1}}(\Omega)$ with $|G|^2 = h^{-1}$ on $\Omega.$ Now consider the map  $M_F: \big(\mathbb A^2(\Omega), h dv\big) \to \big(\mathbb A^2_{\gamma}(\Omega),dv\big) ,$ defined by the equation 
\begin{align*}
M_F(g) =F g, \;\;\;\;g\in \big(\mathbb A^2(\Omega), h dv\big). 
\end{align*}
Clearly, $M_F$ is a unitary operator and its inverse is the operator $M_{G}.$ The multiplication operator $M_F$ intertwines the corresponding operator of multiplication by the coordinate function on the Hilbert spaces $\big(\mathbb A^2(\Omega), h dv\big)$ and $\big(\mathbb A^2_{\gamma}(\Omega),dv\big).$ The character $\gamma$ is determined by $\gamma_j(h) = \exp(ic_j(h)),$  where $c_j(h)$ is given by 
\begin{align*}
 c_j(h) = -\int _{\partial \Omega _{j}} \frac{\partial}{\partial \eta _z}\big(\frac{1}{2}\log h(z)\big)ds_z, \;\;\;\mbox{for}\; j= 1,2,...,n.\;\;
\end{align*}
  
 Conversely, following second part of the lemma \ref{existence of modulus auto}, for any character $\gamma$ there exist a positive function $h$ on $\overbar{\Omega}$ such that $\log h$ is harmonic on $\overbar{\Omega}$ and $h = |F|^2$ on $\overbar{\Omega}$ for some function $F$ in $H^{\infty}_{\gamma}(\Omega).$ Thus we have established a surjective map from the class $ \mathcal G = \{M \, \text{on}\, (\mathbb A^2(\Omega),  hdv)$:  $\log h$ is harmonic on $\overbar{\Omega} \}$  onto the class of Bergman bundle shift of rank 1, namely, $\{M \, \text{on}\, \big(\mathbb A^2_{\gamma}(\Omega),dv\big)$: $\gamma \in \text{Hom}(\pi_1(\Omega), \mathbb S^1)\}.$
 Thus we have proved the  theorem stated below.  
\begin{thm}
There is a bijective correspondence between the multiplication operators on the weighted Bergman spaces $\mathcal G$ and the bundle shifts $\{M \text{ on } \big(\mathbb A^2_{\gamma}(\Omega),dv\big)$: $\gamma \in \text{Hom}(\pi_1(\Omega), \mathbb S^1)\}.$
 \end{thm}
Also, the following corollary is an immediate consequence of  \cite[Theorem 18]{Bergmanbundle}.
\begin{cor}
Let $h_1, h_2$ be two positive function on $\overbar{\Omega}.$ Suppose that $\log h_i,$ $i=1,2,$ are harmonic on $\overbar{\Omega}.$ Then the operator $M$ on $(\mathbb A^2(\Omega),  h_1 dv)$ is unitarily equivalent to the operator $M$ on $(\mathbb A^2(\Omega),  h_2 dv)$ if and only if $\gamma_j(h_1) = \gamma_j(h_2)$ for $j=1,2,\ldots,n.$ 
\end{cor}

\bibliographystyle{amsplain}\bibliography{bibfile}

\end{document}